\documentclass[11pt]{article}
\usepackage{amsmath}
\usepackage[thmmarks,amsmath,amsthm]{ntheorem}

\usepackage{amsfonts, psfrag, graphicx, amssymb,enumitem,indentfirst,float,geometry,color,enumerate,graphicx,subfigure,algorithm,algpseudocode,pifont,hyperref}
\usepackage{caption}

\allowdisplaybreaks[4]
\numberwithin{equation}{section}
\numberwithin{figure}{section}

\usepackage{lipsum}
\newcommand\blfootnote[1]{%
  \begingroup
  \renewcommand\thefootnote{}\footnote{#1}%
  \addtocounter{footnote}{-1}%
  \endgroup
}

\newtheorem{thm}{Theorem}[section]
\newtheorem{lemma}{Lemma}[section]
\newtheorem{rem}{Remark}[section]

\newcommand{\commentout}[1]{{}} 

\newcommand{\abs}[1]{\left|#1\right|}
\newcommand{\norm}[1]{\left\|#1\right\|}

\newcommand{\bfc}{{\bf c}}

\newcommand{\bfu}{{\bf u}}

\newcommand{\bfv}{{\bf v}}

\newcommand{\bfphi}{\boldsymbol{\phi}}

\newcommand{\bfpsi}{\boldsymbol{\psi}}

\begin{document}
\title{Approximation Capabilities of Immersed Finite Element Spaces\\
for Elasticity Interface Problems \thanks{This research was partially supported by GRF B-Q56D and B-Q40W of HKSAR.}
\blfootnote{Keywords: Interface problems, discontinuous coefficients, elasticity equations, immersed finite element method, .}
}
\author{Ruchi Guo \thanks{ruchi91@vt.edu, Department of Mathematics,
Virginia Tech}
\and Tao Lin \thanks{tlin@vt.edu, Department of Mathematics,
Virginia Tech}
\and Yanping Lin \thanks{yanping.lin@polyu.edu.hk, Department of Applied Mathematics, Hong Kong Polytechnic University}}
\date{}
\maketitle

\begin{abstract}
We construct and analyze a group of immersed finite element (IFE) spaces formed by linear, bilinear and rotated $Q_1$ polynomials for solving planar elasticity equation involving interface. The shape functions in these IFE spaces are constructed through a group of approximate jump conditions such that the unisolvence of the bilinear and rotated $Q_1$ IFE shape functions are always guaranteed regardless of the Lam\'e parameters and the interface location. The boundedness property and a group of identities of the proposed IFE shape functions are established. A multi-point Taylor expansion is utilized to show the optimal approximation capabilities for the proposed IFE spaces through the Lagrange type interpolation operators.
\end{abstract}

\section{Introduction}

In many applications of sciences and engineering, we need to consider an elastic object formed with multiple materials which leads to a linear elasticity system with discontinuous coefficients whose values reflect the difference between materials. To be specific and without loss of generality, we consider an elastic object forming a domain $\Omega \subset \mathbb{R}^2$ separated by a smooth interface curve $\Gamma$ into two sub-domains $\Omega^-$ and $\Omega^+$ each of which is occupied by a different material. Thus, the Lam\'{e} parameters of $\Omega$ are piecewise constant functions in the following forms:
\begin{equation}
\label{lame_para}
\lambda=
\left\{\begin{array}{cc}
\lambda^- & \text{if} \; X\in \Omega^- , \\
\lambda^+ & \text{if} \; X\in \Omega^+ ,
\end{array}\right.
~~~~~~~~~~
\mu=
\left\{\begin{array}{cc}
\mu^- & \text{if} \; X\in \Omega^- , \\
\mu^+ & \text{if} \; X\in \Omega^+ .
\end{array}\right.
\end{equation}
As usual, we assume that the $\mathbf{ u}=(u_1,u_2)^T$ is modeled by the planar linear elasticity equations:
\begin{eqnarray}
\label{elas_eq0}
&- \textrm{div} ~ \sigma(\mathbf{ u}) = \mathbf{ f} ~~~ &\textrm{in}~\Omega^- \cup \Omega^+, \\
& \mathbf{ u} = \mathbf{ g} ~~ &\textrm{on} ~ \partial\Omega,
\end{eqnarray}
where $\mathbf{ f}=(f_1,f_2)^T$ and $\mathbf{ g}=(g_1,g_2)^T$ represent the given body force and the displacement on the boundary, respectively, $\sigma(\mathbf{ u})=(\sigma_{ij}(\mathbf{ u}))_{1\leqslant i,j \leqslant 2}$ is the stress tensor given by
\begin{equation}
\label{stress_tensor}
\sigma_{ij}(\mathbf{ u}) = \lambda (\nabla\cdot \mathbf{ u}) \delta_{i,j} + 2\mu \epsilon_{ij}(\mathbf{ u}), ~~~ \textrm{with}~~\epsilon_{ij}(\mathbf{ u})=\frac{1}{2}\left( \frac{\partial u_i}{\partial x_j} + \frac{\partial u_j}{\partial x_i} \right)
\end{equation}
being the strain tensor. Furthermore, the discontinuity in the Lam\'{e} parameters requires the displacement $\mathbf{ u}=(u_1,u_2)^T$ to satisfy the jump conditions across the material interface $\Gamma$:
\begin{eqnarray}
 & [\mathbf{  u}]_{\Gamma} &:= (\mathbf{ u}^+-\mathbf{ u}^-)|_{\Gamma} = \mathbf{ 0}, \label{jump_contin} \\
 & [\sigma(\mathbf{ u})\mathbf{ n}]_{\Gamma} &:= (\sigma^+(\mathbf{ u}^+)\mathbf{ n} - \sigma^-(\mathbf{ u}^-)\mathbf{ n})|_{\Gamma}  = \mathbf{ 0}, \label{jump_stress}
\end{eqnarray}
where $\bfu^s = \bfu|_{\Omega^s}, \sigma^s(\mathbf{ u}^s) = \sigma(\mathbf{ u})|_{\Omega^s}, s = -, +$, and $\mathbf{ n}$ is the normal vector to $\Gamma$.

We call \eqref{elas_eq0}-\eqref{jump_stress} an elasticity interface problem for determining the displacement $\mathbf{ u}=(u_1,u_2)^T$.
Elasticity interface problems have a wide range of applications in engineering and science, such as the inverse problems \cite{1999AbdaAmeur,1999AlvesTuong,2004AmeurBurger} in which one needs to recover the location or geometry of buried cracks, cavities or inclusions, and the structure optimization problems \cite{2003BendsoeSigmund,140Cherkaev,J.Sokolowski_J.-P.Zolesio_1992} in which one aims at optimizing the distribution of different elastic materials such that the overall structure compliance can be minimized, and additional elasticity problems can be found in \cite{2001GaoHuangAbrham,1997JouLeoLowengrub,2000LeoLowengrubNie,1995SuttonBalluffi}, to name just a few.

\commentout{
Elasticity interface problems have wide application in engineering and science. 
For instance, a large group of inverse problems arise in the context of linear elasticity involving interface, in which one needs to either recover the location or geometry of buried cracks, cavities or inclusions \cite{1999AbdaAmeur,1999AlvesTuong,2004AmeurBurger} with different elastic properties than the surroundings or identify the unknown parameters of various elastic materials \cite{1995Andrei,2008JadambaRaciti,2013Hegemann}. Another field of great research interest is the elastic structure optimization problem. In this field, one aims at optimizing the distribution of materials with different elastic properties such that overall structure compliance can be minimized \cite{2003BendsoeSigmund,140Cherkaev,J.Sokolowski_J.-P.Zolesio_1992}. And for other interesting elastic problems, we refer readers to \cite{2001GaoHuangAbrham,1997JouLeoLowengrub,2000LeoLowengrubNie,1995SuttonBalluffi}, to name just a few.
}

Finite element methods \cite{2008BrennerScott,1988Ciarlet,2000ZienkiewiczTaylor} and discontinuous Galerkin methods \cite{2006CockburnSchotzauWang,2003HansboLarson,2006Wihler} have been developed to solve elasticity interface problems, and these methods perform optimally provided that their mesh is interface-fitted \cite{2000BabuskaOsborn,1998ChenZou}. In some applications, such as those inverse/design problems mentioned above, the shape or location of the interface is usually changing in the computation. And in general, it is non-trivial and time consuming to generate an interface-fitting mesh again and again; therefore, solving \eqref{elas_eq0}-\eqref{jump_stress}
on interface-independent (non-interface-fitted) meshes has attracted research attentions.
Both the finite element approach or the finite difference approach have been attempted. For example,
a unfitted finite element method using the Nitsche's penalty along the interface to enforce the jump conditions is presented in \cite{2009BeckerBurmanHansbo,2004HansboHansbo}, and some immersed interface methods based on finite difference formulation are presented in
\cite{2004YangTHESIS,2003YangLiLi} which handle the jump conditions through a local coordinate transformation between sub-elements partitioned by the interface.

Immersed finite element (IFE) methods are developed for solving interface problems with interface-independent meshes.
The key idea of an IFE space is to use standard polynomials on non-interface elements, but Hsieh-Clough-Tocher type \cite{2001Braess,1966CloughTocher} macro polynomials constructed according to interface jump conditions on interface elements. There have been quite a few publications on IFE methods, for example, IFE methods for elliptic interface problems are discussed in \cite{2008HeLinLin,2011HeLinLin,2004LiLinLinRogers,2003LiLinWu,2015LinLinZhang,2016AdjeridGuoLin,2013ZhangTHESIS}, IFE methods for
interface problems of other types partial differential equations are presented in \cite{2015AdjeridChaabaneLin,2017ChuHanCao,2013HeLinLinZhang,2011LinLinSunWang,2016Moon,2015ChenwuXiao}. In particular, for planar-elasticity interface problems
described by \eqref{elas_eq0}-\eqref{jump_stress}, a non-conforming linear IFE space on a uniform triangular mesh is discussed in
in \cite{2007GongTHESIS,2010GongLi,2004YangTHESIS}. A conforming IFE space is developed in \cite{2005LiYang,2004YangTHESIS} by extending the
IFE shape functions in \cite{2007GongTHESIS,2010GongLi,2004YangTHESIS} to the neighborhood interface elements. A bilinear
IFE space on a rectangular Cartesian mesh is discussed in \cite{2012LinZhang}. A non-conforming IFE space using the rotated $Q_1$ polynomials
is presented in \cite{2013LinSheenZhang} which leads to a locking-free IFE method. Most of the IFE methods for the planar-elasticity interface problems
are Gelerkin type, i.e., the test and trial functions used in each of these methods are from the same IFE space, but the Petrov-Galerkin formulation
can also be used, for example, a Petrov-Galerkin IFE scheme is developed in \cite{2012HouLiWangWang} that uses standard Lagrange polynomials as the test functions.

\commentout{
The immersed finite element(IFE) methods \cite{2008HeLinLin,2011HeLinLin,2004LiLinLinRogers,2015LinLinZhang,2016AdjeridGuoLin,2013ZhangTHESIS} were developed for the elliptic interface problems by employing the Hsieh-Clough-Tocher type macro functions \cite{2001Braess} constructed according to the jump conditions on interface elements to capture the jump behavior. See \cite{2015AdjeridChaabaneLin,2017ChuHanCao,2013HeLinLinZhang,2011LinLinSunWang,2016Moon,2015ChenwuXiao} for the application of IFE methods on interface problems with other equations or jump conditions. Specifically for the elasticity interface problems, a non-conforming linear IFE space was first proposed in \cite{2007GongTHESIS,2010GongLi,2004YangTHESIS} on a uniform triangular mesh. This idea was then employed to construct conforming IFE spaces by extending the shape functions to the neighborhood interface elements \cite{2005LiYang,2004YangTHESIS}. A bilinear
IFE space on a rectangular Cartesian mesh is discussed in \cite{2012LinZhang} which also presents special configuration of an interface element and
Lam\'e parameters in which the Lagrange type linear or bilinear IFE shape functions cannot be constructed.

Besides, instead of using the IFE shape functions as the test functions, the authors in proposed the Petrov Galerkin scheme by employing the standard Lagrange polynomials as the test functions. We note that these IFE spaces for the elasticity interface problems in the literature lack of theoretical analysis for even their approximation capabilities.

Recently in \cite{2012LinZhang}, the authors designed some special interface element configuration on which the linear system for solving the linear or bilinear IFE shape functions is singular for specific Lam\'e parameters. This non-unisolvence was circumvented in \cite{2013LinSheenZhang} by using the IFE spaces constructed from the rotated $Q_1$ polynomials which is also locking free. Besides, instead of using the IFE shape functions as the test functions, the authors in \cite{2012HouLiWangWang} proposed the Petrov Galerkin scheme by employing the standard Lagrange polynomials as the test functions. We note that these IFE spaces for the elasticity interface problems in the literature lack of theoretical analysis for even their approximation capabilities.
}

This article focuses on two issues in the research of IFE methods for interface problems of the planar elasticity. First, we develop a
unified construction procedure for shape functions of the IFE spaces defined on a triangular or rectangular mesh with linear or bilinear or rotated $Q_1$ polynomials, respectively. By this procedure, the coefficients in an IFE shape function satisfy a \textit{Sherman-Morrison} linear system from which the unisolvance of the IFE shape functions can be readily deducted. Second, we derive a group of multi-point Taylor expansions for vector functions satisfying the jump conditions specified in \eqref{jump_contin} and \eqref{jump_stress} for the planar-elasticity interface problems, and we then employ them in the framework recently developed in \cite{2016GuoLin} to show the optimal approximation capabilities for the IFE spaces considered in this article. However, in contrast to
the usual Lagrange type finite element space for the planar elasticity boundary value problems, the two components in each vector IFE shape function on an interface element are coupled by the interface jump conditions; hence, the error analysis presented in this article has some essential features different from those for scalar IFE spaces discussed in \cite{2016GuoLin,2008HeLinLin,2004LiLinLinRogers}. To our best knowledge,
this is the first time the approximation capability of IFE spaces formed with vector functions is analyzed, and this is an important step towards to the
establishment of the theoretical foundation for IFE methods that can solve interface problems of the linear elasticity system with interface-independent (such as Cartesian) meshes.

\commentout{
The contribution of this article consists of two parts. First we develop a \textit{Sherman-Morrison} linear system to construct IFE shape functions as piecewise linear or bilinear polynomials by enforcing the continuity and stress jump conditions on the line connecting the intersection points of the interface with the element edges. We show that for the linear case the proposed IFE shape functions are equivalent to those in the literature \cite{2012LinZhang} but for the bilinear and rotated $Q_1$ cases, the proposed IFE shape functions can be always uniquely determined by their values at the element vertices or midpoints of edges regardless of the interface location and Lam\'e parameters. Secondly we derive a group of multi-point Taylor expansion for vector functions satisfying the continuity and stress jump conditions and apply them in the framework developed by \cite{2016GuoLin} to show the optimal approximation capabilities for the proposed IFE spaces thorough the Lagrange type interpolation operator. To our best knowledge, applying the multi-point Taylor expansion to handle elasticity jump conditions is rarely seen in the literature.
}

This article consists of 5 additional sections. The next section is for some basic notations and assumptions. In Section 3, we establish a few fundamental geometric identities and estimates related to the interface. In Section 4, we derive the multi-point Taylor expansions for a vector function $\mathbf{ u}$ satisfying the jump conditions \eqref{jump_contin} and \eqref{jump_stress} along the interface. In Section 5, we derive a \textit{Sherman-Morrison} linear system for determining the coefficients in IFE shape functions on an interface element, study properties of these shape functions, and prove the optimal approximation capabilities for the IFE spaces considered in this article. In the last section, we present a group of numerical examples to illustrate the approximation features of these IFE spaces.


\section{Preliminaries}\label{sec:preliminaries}
We now describe terms and facts to be used in the discussions. Let $\Omega\subset \mathbb{R}^2$ be a bounded domain formed as union of finitely many rectangles/triangles, and without loss of generality, we assume $\Omega$ is separated by $\Gamma$ into two subdomains $\Omega^+$ and $\Omega^-$ such that $\overline{\Omega} = \overline{\Omega^+} \cup \overline{\Omega^-}$. For a measurable subset $\widetilde{\Omega}\subseteq\Omega$, we define the vector Sobolev space
$\mathbf{ W}^{k,p}(\widetilde{\Omega})= \left[W^{k,p}(\widetilde{\Omega})\right]^2$
where $W^{k,p}(\widetilde{\Omega})$ is the standard Sobolev space, and the associated norm and semi-norm of $\mathbf{ W}^{k,p}(\widetilde{\Omega})$ are such that
for every $\mathbf{ u} = (u_1, u_2)^T \in W^{k,p}(\widetilde{\Omega})$,
\begin{equation}
\label{sobolev_norm}
\| \mathbf{ u} \|_{k,p,\widetilde{\Omega}}=\| u_1 \|_{k,p,\widetilde{\Omega}} + \| u_2 \|_{k,p,\widetilde{\Omega}} ~~~\textrm{and} ~~~ |\mathbf{ u}|_{k,p,\widetilde{\Omega}}= \| D^{\alpha}u_1 \|_{0,p,\widetilde{\Omega}} +  \| D^{\alpha}u_2 \|_{0,p,\widetilde{\Omega}}, ~~ |\alpha|=k.
\end{equation}
The related vector Hilbert space is denoted by $\mathbf{ H}^k(\widetilde{\Omega})=\mathbf{ W}^{k,2}(\widetilde{\Omega})$. Let $\mathbf{ C}^k(\widetilde{\Omega})$ be the collection of $k$-th differentiable smooth vector functions. When $\widetilde{\Omega}^s=\widetilde{\Omega}\cap\Omega^s \not = \emptyset, s = \pm$, and
$k \geq 1$, we define
\begin{equation}
\label{Hil_sp_int}
\mathbf{ PW}^{k, p}_{int}(\widetilde{\Omega}) = \{ \mathbf{ u}\,:\, \mathbf{ u}\in \mathbf{ W}^{k,p}(\widetilde{\Omega}^s),~s=\pm; ~ [\mathbf{ u}]_{\Gamma}=\mathbf{ 0},~\textrm{and} ~ [\sigma(\mathbf{ u})\mathbf{ n}]_{\Gamma}=\mathbf{ 0} \},
\end{equation}
\begin{equation}
\label{C_sp_int}
\mathbf{ PC}^k_{int}(\widetilde{\Omega}) = \{ \mathbf{ u}\,:\, \mathbf{ u}\in \mathbf{ C}^k(\widetilde{\Omega}^s),~s=\pm; ~ [\mathbf{ u}]_{\Gamma}=\mathbf{ 0},~\textrm{and} ~ [\sigma(\mathbf{ u})\mathbf{ n}]_{\Gamma}=\mathbf{ 0} \},
\end{equation}
with the following norms and semi-norms:
\begin{equation*}
\begin{split}
\label{H_norms}
\| \mathbf{ u} \|_{k,p,\widetilde{\Omega}} = \sum_{i=1}^2 \left( \| u_i \|_{k,p,\widetilde{\Omega}^-} + \| u_i \|_{k,p,\widetilde{\Omega}^+} \right), ~~~~ &\textrm{and} ~~~~
| \mathbf{ u} |_{k,p,\widetilde{\Omega}} = \sum_{i=1}^2 \left( | u_i |_{k,p,\widetilde{\Omega}^-} + | u_i |_{k,p,\widetilde{\Omega}^+} \right), \\
\| \mathbf{ u} \|_{k,\infty,\widetilde{\Omega}} =  \max_{i=1,2} \{ \max\{ \| u_i \|_{k,\infty,\widetilde{\Omega}^-} , \| u_i \|_{k,\infty,\widetilde{\Omega}^+} \} \}, ~~~~ &\textrm{and} ~~~~
| \mathbf{ u} |_{k,\infty,\widetilde{\Omega}} =  \max_{i=1,2} \{ \max\{ | u_i |_{k,\infty,\widetilde{\Omega}^-} , | u_i |_{k,\infty,\widetilde{\Omega}^+} \} \}.
\end{split}
\end{equation*}
Also we denote the corresponding Hilbert space $\mathbf{ PH}^k(\widetilde{\Omega})=\mathbf{ PW}^{k,2}(\widetilde{\Omega})$ with the norm $\|\cdot\|_{k,\widetilde{\Omega}}=\|\cdot\|_{k,2,\widetilde{\Omega}}$ and the semi-norm $|\cdot|_{k,\widetilde{\Omega}}=|\cdot |_{k,2,\widetilde{\Omega}}$. Furthermore, for any vector function $\mathbf{ v}=(v_1,v_2)^T\in\mathbf{ H}^1(\widetilde{\Omega})$, let $\nabla\mathbf{ v}$ be its $2$-by-$2$ Jacobian matrix where the $i$-th row is the row vector $\nabla v_i$, $i=1,2$.

Let $\mathcal{T}_h$ be a Cartesian rectangular or triangular mesh of the domain $\Omega$ with a mesh size $h>0$. An element $T\in \mathcal{T}_h$ is called an interface element if the intersection of the interior of $T$ with the interface $\Gamma$ is non-empty; otherwise, it is called a non-interface element. Let $\mathcal{T}^i_h$ and $\mathcal{T}^n_h$ be the sets of interface elements and non-interface elements, respectively. Similarly, let $\mathcal{E}^i_h$ and $\mathcal{E}^n_h$ be the sets of interface edges and non-interface edges, respectively. In addition, as in \cite{2016GuoLinZhang,2009HeLinLin}, we assume that $\mathcal{T}_h$ satisfies the following hypotheses when the mesh size $h$ is small enough:

\begin{itemize}[leftmargin=30pt]
  \item [\textbf{(H1)}] The interface $\Gamma$ cannot intersect an edge of any element at more than two points unless the edge is part of $\Gamma$.
  \item [\textbf{(H2)}] If $\Gamma$ intersects the boundary of an element at two points, these intersection points must be on different edges of this element.
  \item [\textbf{(H3)}] The interface $\Gamma$ is a piecewise $C^2$ function, and the mesh $\mathcal{T}_h$ is formed such that the subset of $\Gamma$ in every interface element $T\in\mathcal{T}^i_h$ is $C^2$.
  \item [\textbf{(H4)}] The interface $\Gamma$ is smooth enough so that $\textbf{PC}^2_{int}(T)$ is dense in $\textbf{PH}^2_{int}(T)$ for every interface element $T\in\mathcal{T}^i_h$.
\end{itemize}

We will discuss IFE spaces formed by linear polynomials on triangular meshes and bilinear or rotated $Q_1$ polynomials on rectangular meshes. For each element $T$ in a mesh $\mathcal{T}_h$, we introduce an index set $\mathcal{I}=\{1,2,3\}$ when $T$ is triangular or $\mathcal{I}=\{1,2,3,4\}$ when $T$ is rectangular. Then, the local finite element space is denoted by $(T,\mathbf{ \Pi}_T,\mathbf{ \Sigma}_T)$, with $\mathbf{ \Pi}_T=\left[ \textrm{Span}\{ 1,x,y \} \right]^2$, $\left[ \textrm{Span}\{ 1,x,y,xy \} \right]^2$ or $\left[ \textrm{Span}\{ 1,x,y,x^2-y^2 \} \right]^2$ for the linear, bilinear or rotated $Q_1$ polynomial space, respectively, and the local degrees of freedom $\mathbf{ \Sigma}_T= \{ \bfpsi_T(A_i)\,:\, i\in\mathcal{I}, ~ \bfpsi_T\in \mathbf{ \Pi}_T \}$, where $A_i$s are vertices of $T$ for the linear and bilinear cases, or midpoints of edges of $T$ for the rotated $Q_1$ case. For these finite element spaces, according to \cite{2008BrennerScott,1978Ciarlet, 1973CrouzeixRaviart,1992RanacherTurek}, there exist vector shape functions $\bfpsi_{i,T}\in \mathbf{ \Pi}_T$, $i=1,2,\cdots,2|\mathcal{I}|$ such that
$\mathbf{ \Pi}_T = \textrm{Span}\{\bfpsi_{i,T}, 1\leq i \leq 2|\mathcal{I}|\}$ with
\begin{equation}
\label{vec_bas}
\bfpsi_{i,T}(A_j)=
\left\{\begin{array}{cc}
\delta_{i,j},  \\
0,
\end{array}\right.
~i=1,\cdots,|\mathcal{I}|,
~~~~ \textrm{and} ~~~~
\bfpsi_{i,T}(A_j)=
\left\{\begin{array}{cc}
0,  \\
\delta_{i-|\mathcal{I}|,j},
\end{array}\right.
~ i=|\mathcal{I}|+1,\cdots,2|\mathcal{I}|,
\end{equation}
\begin{equation}
\label{fe_bound}
| \bfpsi_{i,T} |_{k,\infty,T} \leqslant Ch^{-k}, ~~k=0,1, 2.
\end{equation}
In addition, we will use a vectorization map \textrm{Vec}\,:\,$\mathbb{R}^{m\times n} \rightarrow \mathbb{R}^{mn\times 1}$ such that for any $A=(a_{ij})^{m,n}_{i=1,j=1}$,
$$
\textrm{Vec}(A):=(a_{11},\cdots,a_{m1},a_{12},\cdots,a_{m2},\cdots,a_{1n},\cdots,a_{mn})^T,
$$
and a Kronecker product $\otimes\,:\, \mathbb{R}^{m\times n}\times\mathbb{R}^{p\times q} \rightarrow \mathbb{R}^{mp\times nq}$ such that for any $A=(a_{ij})^{m,n}_{i=1,j=1}\in\mathbb{R}^{m\times n}$ and $B\in\mathbb{R}^{p\times q}$, there holds
\begin{equation}
\label{Kron_prod}
A \otimes  B =
\left[\begin{array}{ccc}
a_{11}B & \cdots & a_{1n}B \\
\vdots    & \ddots & \vdots \\
a_{m1}B & \cdots & a_{mn}B
\end{array}\right].
\end{equation}
A well-known formula \cite{2005AbadirMagnus} about the Kronecker product and the vectorization operation is the following:
\begin{equation}
\label{kro_vec}
\textrm{Vec}(CDE)=(E^T\otimes C)\textrm{Vec}(D).
\end{equation}
Throughout this article, we use the notation $I_{n}$ to denote the $n$-by-$n$ identity matrix and $0_{m\times n}$ to denote the $m$-by-$n$ zero matrix for any integer $m$ and $n$, and to simplify the presentation, we adopt the notation $\partial_{x_k}=\frac{\partial}{\partial_{x_k}}$, $k=1,2$, for partial derivatives with
$x_1 = x, ~x_2 = y$. Also, as usual, we will use $C$ to denote generic constants independent with the mesh size $h$ in all the discussions from now on.


\section{Geometric Properties of the Interface}
In this section, we derive a group of geometric properties on the interface elements for estimating interpolation errors of vector-valued functions.
These properties are extensions of those developed in \cite{2016GuoLin} for scalar functions. Let $T$ be an interface element and $l$ be a line connecting the intersection points of the interface $\Gamma$ with $\partial T$. Let $\mathbf{ n}(\widetilde{X})=(\tilde{n}_1(\widetilde{X}),\tilde{n}_2(\widetilde{X}))^T$ and $\mathbf{ t}(\widetilde{X})=(\tilde{n}_2(\widetilde{X}),-\tilde{n}_1(\widetilde{X}))^T$ be the normal and tangential vectors of $\Gamma$ at
a point $\widetilde{X}\in\Gamma\cap T$, respectively, and let the normal and tangential vectors of $l$ be $\bar{\mathbf{ n}}=(\bar{n}_1,\bar{n}_2)^T$ and $\bar{\mathbf{ t}}=(\bar{n}_2,-\bar{n}_1)^T$, respectively. Consider the following matrices:
\begin{eqnarray}
&N^s(\widetilde{X}) =&
\left[\begin{array}{cccc}
(\lambda^s + 2\mu^s)\tilde{n}_1(\widetilde{X}) & \mu^s \tilde{n}_2(\widetilde{X}) & \mu^s \tilde{n}_2(\widetilde{X}) & \lambda^s \tilde{n}_1(\widetilde{X}) \\
\lambda^s\tilde{n}_2(\widetilde{X}) & \mu^s\tilde{n}_1(\widetilde{X}) & \mu^s\tilde{n}_1(\widetilde{X}) & (\lambda^s + 2\mu^s)\tilde{n}_2(\widetilde{X}) \\
-\tilde{n}_2(\widetilde{X}) & 0 & \tilde{n}_1(\widetilde{X}) & 0 \\
0 & -\tilde{n}_2(\widetilde{X}) & 0 & \tilde{n}_1(\widetilde{X})
\end{array}\right], ~~~ s=\pm, \\  \label{Nbar_s}
&\overline{N}^s =&
\left[\begin{array}{cccc}
(\lambda^s + 2\mu^s)\bar{n}_1 & \mu^s \bar{n}_2 & \mu^s \bar{n}_2 & \lambda^s \bar{n}_1 \\
\lambda^s\bar{n}_2 & \mu^s\bar{n}_1 & \mu^s\bar{n}_1 & (\lambda^s + 2\mu^s)\bar{n}_2 \\
-\bar{n}_2 & 0 & \bar{n}_1 & 0 \\
0 & -\bar{n}_2 & 0 & \bar{n}_1
\end{array}\right], ~~~ s=\pm.  \label{N_s}
\end{eqnarray}
By straightforward calculation, we have
\begin{equation}
\label{M_det}
\textrm{Det}(N^s(\widetilde{X}))=\textrm{Det}(\overline{N}^s)=\mu^s(\lambda^s+2\mu^s), ~~ s=\pm.
\end{equation}
Hence both the matrices $N^s(\widetilde{X})$ and $\overline{N}^s$ are non-singular, and we can use them to define
\begin{equation}
\label{M}
M^-(\widetilde{X})= \left( N^+(\widetilde{X}) \right)^{-1}N^-(\widetilde{X}), ~~~~~~ M^+(\widetilde{X})= \left( N^-(\widetilde{X}) \right)^{-1}N^+(\widetilde{X}),
\end{equation}
\begin{equation}
\label{Mbar}
\overline{M}^-= \left( \overline{N}^+ \right)^{-1} \overline{N}^-, ~~~~~~ \overline{M}^+= \left( \overline{N}^-\right)^{-1}\overline{N}^+.
\end{equation}
These matrices have the following properties.

\begin{lemma}
\label{MM_relat}
For every $\mathbf{ u}\in \mathbf{ PC}^2_{int}(T)$ and $\widetilde{X}\in\Gamma\cap T$, there holds
\begin{equation}
\label{Mu}
\emph{Vec}(\nabla \mathbf{ u}^+(\widetilde{X})) = M^-(\widetilde{X})\emph{Vec}(\nabla \mathbf{ u}^-(\widetilde{X})), ~~~~~~
\emph{Vec}(\nabla \mathbf{ u}^-(\widetilde{X})) = M^+(\widetilde{X})\emph{Vec}(\nabla \mathbf{ u}^+(\widetilde{X})).
\end{equation}
\end{lemma}
\begin{proof}
To simplify the notations, we denote $\mathbf{ n}(\widetilde{X})=(\tilde{n}_1,\tilde{n}_2)^T$ in this proof. By direction calculations, we have
\begin{equation}
\label{Mu1}
\sigma^s(\mathbf{ u}(\widetilde{X})) \, \mathbf{ n}(\widetilde{X})=
\left[\begin{array}{c}
(\lambda^s + 2\mu^s)\tilde{n}_1 \partial_{ x_1} u_1 + \mu^s \tilde{n}_2 \partial_{ x_2} u_1 + \mu^s \tilde{n}_2 \partial_{ x_1} u_2 + \lambda^s \tilde{n}_1 \partial_{ x_2} u_2 \\
\lambda^s\tilde{n}_2 \partial_{x_1} u_1 + \mu^s\tilde{n}_1 \partial_{x_2} u_1 + \mu^s\tilde{n}_1 \partial_{x_1} u_2 + (\lambda^s + 2\mu^s)\tilde{n}_2 \partial_{x_1} u_2
\end{array}\right].
\end{equation}
From the continuity jump condition \eqref{jump_contin}, we have $\nabla u^+_i\, \mathbf{ t}(\widetilde{X}) = \nabla u^-_i\, \mathbf{ t}(\widetilde{X})$, $i=1,2$. Combining this with the stress jump condition \eqref{jump_stress} leads to
$
N^-(\widetilde{X}) \textrm{Vec}(\nabla \mathbf{ u}^-(\widetilde{X})) = N^+(\widetilde{X}) \textrm{Vec}(\nabla \mathbf{ u}^+(\widetilde{X}))
$
from which we have \eqref{Mu} because of \eqref{M}.
\commentout{
through the following matrix form:
Combining it with the matrix form for the continuity jump condition \eqref{jump_contin}, $\nabla u^+_i\cdot \mathbf{ t}(\widetilde{X}) = \nabla u^-_i\cdot \mathbf{ t}(\widetilde{X})$, $i=1,2$, we have
$
N^-(\widetilde{X}) \textrm{Vec}(\nabla \mathbf{ u}^-(\widetilde{X})) = N^+(\widetilde{X}) \textrm{Vec}(\nabla \mathbf{ u}^+(\widetilde{X}))
$
which yields \eqref{Mu} because of \eqref{M}.
}
\end{proof}

\begin{lemma}
\label{eigen_vec}
The vectors $\alpha_1=[-\bar{n}_2,0,\bar{n}_1,0]^T$ and $\alpha_2=[0,-\bar{n}_2,0,\bar{n}_1]^T$ are eigenvectors of $\left( \overline{M}^s \right)^T$, $s=+$, or $-$, such that
\begin{equation}
\label{M_eigen}
\left( \overline{M}^s \right)^T\alpha_i = \alpha_i, ~~ i=1,2.
\end{equation}
\end{lemma}
\begin{proof}
The identities in \eqref{M_eigen} follow from direct calculations.
\end{proof}

As proved in the following lemma, the matrices $\overline{M}^s$ constructed on $l$ can be used to approximate the matrices $M^s$ constructed on the interface $\Gamma \cap T$, $s=+$ or $-$.

\begin{lemma}
There exists a constant $C$ independent of the interface location such that for every
interface element $T\in\mathcal{T}_h^i$ and every point $\widetilde{X}\in\Gamma\cap T$, $s=\pm$, we have
\begin{equation}
\label{M_bound}
\| M^s(\widetilde{X}) \| \leqslant C, ~~~~~~ \| \overline{M}^s \| \leqslant C,
\end{equation}
and
\begin{equation}
\label{M_aprox}
\| M^s(\widetilde{X}) - \overline{M}^s \| \leqslant Ch.
\end{equation}
\end{lemma}
\begin{proof}
We only prove the case for $s=-$, and the argument for $s=+$ is similar. Since $\| \bar{\mathbf{ n}} \|=1$ and $\|\mathbf{ n}(\widetilde{X}) \| = 1$, we have $\| \overline{N}^- \| \leqslant C$ and $\| N^-(\widetilde{X}) \|\leqslant C$. Besides, we note that
$$
\| \left( N^+(\widetilde{X}) \right)^{-1} \| = \frac{1}{\textrm{Det}(N^+(\widetilde{X}))} \| \textrm{adj}(N^+(\widetilde{X})) \| \leqslant C, ~~~~\textrm{and} ~~~~ \| \left( \overline{N}^+ \right)^{-1} \| = \frac{1}{\textrm{Det}(\overline{N}^+) } \| \textrm{adj}(\overline{N}^+) \| \leqslant C,
$$
because $\textrm{Det}(N^-(\widetilde{X}))=\textrm{Det}(\overline{N}^-)=\mu^-(\lambda^-+2\mu^-)$ and each term of the adjugate matrices is bounded by some constants $C$. Then, \eqref{M_bound} follows from applying these estimates in the inequalities below:
$$
\| M^-(\widetilde{X}) \| \leqslant \| \left( N^+(\widetilde{X}) \right)^{-1} \|  ~ \| N^-(\widetilde{X}) \| ~~~~\textrm{and} ~~~~ \| \overline{M}^- \| \leqslant \| \left( \overline{N}^+ \right)^{-1} \|  ~ \| \overline{N}^- \|.
$$
For \eqref{M_aprox}, we note that
\begin{equation*}
\begin{split}
\| M^-(\widetilde{X}) - \overline{M}^- \| &= \| \left( N^+(\widetilde{X}) \right)^{-1} N^-(\widetilde{X})  -  \left( \overline{N}^+ \right)^{-1} \overline{N}^- \| \\
&= \| \left( N^+(\widetilde{X}) \right)^{-1} \left( N^-(\widetilde{X}) - \overline{N}^- \right) +  \left( N^+(\widetilde{X}) \right)^{-1} \left( \overline{N}^+ - N^+(\widetilde{X}) \right) \left( \overline{N}^+ \right)^{-1}  \overline{N}^-  \| \\
&\leqslant C \| N^-(\widetilde{X}) - \overline{N}^- \| + C \|  \overline{N}^+ - N^+(\widetilde{X}) \| \\
&\leqslant Ch
\end{split}
\end{equation*}
in which we have used the estimate $\| \mathbf{ n}(\widetilde{X}) - \bar{\mathbf{ n}} \| \leqslant Ch$ given in Lemma 3.2 of \cite{2016GuoLin}.
\end{proof}

In addition, we consider the following set
\begin{equation}
\label{T_int}
T_{int} = \bigcup \Big\{ l_t\cap T: l_t~\text{is a tangent line to}~\Gamma\cap T \Big\},
\end{equation}
which is the subelement swept by the tangent lines to $\Gamma\cap T$. We note that this set is equivalent to the one considered in \cite{2016GuoLinZhang}. Following an idea similar to that used in \cite{2016GuoLin}, we can show that
this is actually a small set in the following lemma.

\begin{lemma}
\label{T_int_area}
Assume $h$ is small enough, then there exists a constant $C$ independent of the interface location, such that $|T_{int}|\leqslant Ch^3$.
\end{lemma}
\begin{proof}
Let $\kappa$ be the maximal curvature of $\Gamma\cap T$. By the assumption, we can follow the idea in \cite{2016GuoLin} to assume $h$ is sufficiently small such $\kappa h\leqslant \epsilon$ for some $\epsilon\in(0,1/2)$. Let $D$ and $E$ be the intersection points of $\Gamma$ and $\partial T$, and recall $l$ is the line connecting $D$ and $E$. Let $X$ be one point in $T_{int}$. According to the definition \eqref{T_int}, there exists a point $Y\in \Gamma$ such that $\overline{XY}$ is tangent to $\Gamma$. Denote $X_{\bot}$ and $Y_{\bot}$ as the projection of $X$ and $Y$ onto $l$, respectively. Let $\theta\in [0,\pi/2]$ be the angle between $\overline{XY}$ and $l$. According to Lemma 3.2 in \cite{2016GuoLin}, we have $|\overline{YY_{\bot}}|\leqslant 2(1-2\epsilon^2)^{-3/2}\kappa h^2$. Using the fact $|\overline{XY}|\leqslant Ch$ and simply geometry, we obtain $|\overline{XX_{\bot}}|=|\overline{YY_{\bot}}|+|\overline{XY}|\sin(\theta)\leqslant Ch^2+Ch\sin(\theta)$. In addition, using (3.5b) in \cite{2016GuoLin}, there holds
\begin{equation}
\begin{split}
\label{T_int_area_eq1}
\sin(\theta)&=(1-(\bar{\mathbf{ n}}\cdot\mathbf{ n}(Y))^2)^{1/2} \\
&\leqslant (1+(1-2\epsilon^2)^{-3/2}) \kappa h ( 4 - (1+(1-2\epsilon^2)^{-3/2})^2 \kappa^2 h^2 )^{1/2} \\
&\leqslant (1+(1-2\epsilon^2)^{-3/2}) ( 4 - (1+(1-2\epsilon^2)^{-3/2})^2 \epsilon^2 )^{1/2} \kappa h
\end{split}
\end{equation}
where we have used $h\kappa\leqslant\epsilon$. It shows that $|\overline{XX_{\bot}}|\leqslant C(1+ \kappa) h^2$ with $C$ depending on $\epsilon$, i.e., the distance between $X$ and $\overline{DE}$ is bounded by $C(1+\kappa) h^2$. Since $|\overline{DE}|\leqslant Ch$, we have $|T_{int}|\leqslant C(1+\kappa) h^3$.
\end{proof}

\commentout{

given some prescribed parameters $\epsilon\in(0,\sqrt{2}/2)$ and $\bar{\kappa}\in(0,1]$, we assume the mesh size $h$ is small enough such that
\begin{equation}
\label{mesh_assump}
h < \min\Bigg\{\frac{\sqrt{\bar{\kappa}}}{\sqrt{2}(1+(1-2\epsilon^2)^{-3/2})\kappa} , \frac{\epsilon}{\kappa} \Bigg\}.
\end{equation}
Under the assumption \eqref{mesh_assump}, we can estimate the measure of $T_{int}$.
\begin{lemma}
\label{T_int_area}
Assume the $h$ satisfies \eqref{mesh_assump}, then there exists a constant $C$ independent of the interface location, such that $|T_{int}|\leqslant Ch^3$, where $C$ depends on the maximal curvature $\kappa$.
\end{lemma}
\begin{proof}
Let $D$ and $E$ be the intersection points of $\Gamma$, and $\partial T$ and recall $l$ is the line connecting $D$ and $E$. Let $X$ be one point in $T_{int}$. According to the definition \eqref{T_int}, there exists a point $Y\in \Gamma$ such that $\overline{XY}$ is tangent to $\Gamma$. Denote $X_{\bot}$ and $Y_{\bot}$ as the projection of $X$ and $Y$ onto $l$, respectively. Let $\theta\in [0,\pi/2]$ be the angle between $\overline{XY}$ and $l$. Based on simply geometry, we have $|\overline{XX_{\bot}}|=|\overline{YY_{\bot}}|+|\overline{XY}|\sin(\theta)\leqslant Ch^2+Ch\sin(\theta)$, since $|\overline{YY_{\bot}}|\leqslant 2(1-2\epsilon^2)^{-3/2}\kappa h^2$ and $|\overline{XY}|\leqslant Ch$ according to Lemma 3.2 in \cite{2016GuoLin}. In addition, using (3.5b) in \cite{2016GuoLin}, there holds
\begin{equation}
\begin{split}
\label{T_int_area_eq1}
\sin(\theta)&=(1-(\bar{\mathbf{ n}}\cdot\mathbf{ n}(Y))^2)^{1/2} \\
&\leqslant (1+(1-2\epsilon^2)^{-3/2}) \kappa h \left( 4 - (1+(1-2\epsilon^2)^{-3/2})^2 \kappa^2 h^2 \right)^{1/2} \\
&\leqslant (1+(1-2\epsilon^2)^{-3/2}) \kappa h (4 - \bar{\kappa}/2)^{1/2} \\
&\leqslant C \kappa h,
\end{split}
\end{equation}
where we have used the assumption \eqref{mesh_assump}. It shows that $|\overline{XX_{\bot}}|\leqslant C(1+ \kappa) h^2$, i.e., the distance between $X$ and $\overline{DE}$ is bounded by $C(1+\kappa) h^2$. Since $|\overline{DE}|\leqslant Ch$, we have $|T_{int}|\leqslant C\kappa h^3$.
\end{proof}
}


\section{Multi-point Taylor Expansion}
In this section, we present a multi-point Taylor expansion for the piecewise smooth vector functions satisfying \eqref{jump_contin}-\eqref{jump_stress} on interface elements and derive the estimates of the remainders in the expansion. The multi-point Taylor expansion idea was first employed in \cite{2004LiLinLinRogers} for showing the approximation capabilities of the linear IFE spaces for the elliptic interface problems through the Lagrange type interpolation operator. Similar ideas was then used to study approximation capabilities for the bilinear IFE spaces \cite{2009HeTHESIS,2008HeLinLin} and nonconforming IFE spaces \cite{2016GuoLinZhang,2013ZhangTHESIS}. Recently, by generating this technique, the authors in \cite{2016GuoLin} developed a unified framework to show the approximation capabilities of various IFE spaces, and we now extend this technique to IFE spaces of vector functions
for solving the elasticity interface problems.

In the following discussion, for every $T\in\mathcal{T}^i_h$, let $\Gamma$ partition $T$ into $T^{\pm}$ and let $l$ partition $T$ into $\overline{T}^{\pm}$. Define $\widetilde{T}=(T^+\cap\overline{T}^-)\cup(T^-\cap\overline{T}^+)$ which is the subelement sandwiched by $\Gamma$ and $l$. From \cite{2004LiLinLinRogers}, we know $|\widetilde{T}|\leqslant Ch^3$. In addition, as in \cite{2016GuoLinZhang}, for every $X\in T\backslash T_{int}$, the segment $\overline{A_iX}$  intersects with $\Gamma\cap T$ either at only one point when $A_i$ and $X$ are on different sides of $\Gamma\cap T$ or no point when $A_i$ and $X$ are on the same side. Then, we define $T_{\ast}^s= \overline{T}^s \cap (T^s\backslash T_{int})$, $s=\pm$, and let $T_{\ast}=T\backslash(T^-_{\ast}\cup T^+_{\ast})$.

We further partition $\mathcal{I}$ into two sub index sets $\mathcal{I}^+=\{i\,:\,A_i\in T^+ \}$ and $\mathcal{I}^-=\{i\,:\,A_i\in T^- \}$. For every $X\in T$, we let $Y_i(t,X)=tA_i+(1-t)X$, $t\in [0,1]$, $i\in\mathcal{I}$. Let $\tilde{t}_i=\tilde{t}_i(X)\in [0,1]$ be such that $\widetilde{Y}_i=Y_i(\tilde{t}_i,X)$ is on the curve $\Gamma\cap T$ if $X$ and $A_i$ are on different sides of $T$. We start from the following theorem that gives the expansion of $\mathbf{ u}(A_i)$ about $X$ if $A_i$ and $X$ are the same side of $\Gamma$, i.e., $A_i\in T^s$ and $X\in T^s_{\ast}$, $s=\pm$.
\begin{thm}
For every interface element $T\in\mathcal{T}^i_h$ and $\mathbf{ u}\in\mathbf{ PC}^2_{int}(T)$, assume $A_i \in T^s,~ s = \pm$, then
\begin{equation}
\label{expan_1}
\begin{split}
\mathbf{ u}^s(A_i)
=\mathbf{ u}^s(X)+ \left( (A_i-X)^T\otimes I_2 \right) \emph{Vec}( \nabla\mathbf{ u}^s(X) )+  \mathbf{ R}^s_{i}(X), ~ i\in\mathcal{I}^s, ~ \forall X\in T_{\ast}^s,
\end{split}
\end{equation}
\begin{equation}
\label{expan_1_rem}
\text{where} ~~ \mathbf{ R}^s_i(X)=\int_0^1(1-t)\frac{d^2}{dt^2} \mathbf{ u}^s(Y_i(t,X))dt, ~i\in \mathcal{I}^s, ~\forall X \in T_{\ast}^s.
\end{equation}
\end{thm}
\begin{proof}
Since $A_i \in T^s$ and $X\in T_{\ast}^s$, we know that $Y_i(t, X) \in T^s,~\forall t \in [0, 1]$. Applying the standard Taylor expansion with integral remainder
to the components of $\bfu(X) = (u_1(X), u_2(X))^T$, we have
\begin{equation}
\label{expan_1_1}
\begin{split}
\mathbf{ u}^s(A_i)
=&\mathbf{ u}^s(X)+ \nabla\mathbf{ u}^s(X)(A_i-X) +  \mathbf{ R}^s_{i}(X), ~ i\in\mathcal{I}^s,~\forall X\in T_{\ast}^s, s=\pm.
\end{split}
\end{equation}
Then, we obtain \eqref{expan_1} by applying the vectorization on each side of \eqref{expan_1_1} and using the formula \eqref{kro_vec} with $C=I_2$, $E=A_i-X$ and $D=\nabla \mathbf{ u}^s(X)$.
\end{proof}


In the discussions from now on, we denote $s=\pm$ and $s'=\mp$, which means $s$ and $s'$ always take opposite signs when they appear in the same formula. And in the following theorem, we describe how to expand $\mathbf{ u}(A_i)$ about $X$ if they are the different sides of $\Gamma$, i.e., $A_i\in T^s$ but $X\in T^{s'}_{\ast}$.

\begin{thm}
On every interface element $T\in\mathcal{T}^i_h$ and $\mathbf{ u}\in\mathbf{ PC}^2_{int}(T)$, assume $A_i\in T^{s'}$, then
\begin{equation}
\begin{split}
\label{expan_2}
\mathbf{ u}^{s'}(A_i)=&\mathbf{ u}^{s}(X)+((A_i-X)^T \otimes I_2) \emph{Vec}\left( \nabla\mathbf{ u}^{s}(X) \right)\\
& + ((A_i-\widetilde{Y}_i)^T \otimes I_2) (M^s - I_4) \emph{Vec}\left( \nabla\mathbf{ u}^{s}(X) \right)
+ \mathbf{ R}^s_{i}(X), ~ i\in \mathcal{I}^{s'}, ~ \forall X\in T_{\ast}^s,~ s=\pm,
\end{split}
\end{equation}
where $\mathbf{ R}^s_{i}=\mathbf{ R}^s_{i1}+\mathbf{ R}^s_{i2}+\mathbf{ R}^s_{i3}$, with 
\begin{equation}
\begin{cases}
\label{expan_2_rem}
&\mathbf{ R}_{i1}^{s}(X)=\int_{0}^{\tilde{t}_i}(1-t)\frac{d^2}{dt^2}\mathbf{ u}^s(Y_i(t,X))dt, \\ 
&\mathbf{ R}_{i2}^{s}(X)=\int_{\tilde{t}_i}^{1}(1-t)\frac{d^2}{dt^2}\mathbf{ u}^{s'}(Y_i(t,X))dt, \\ 
&\mathbf{ R}_{i3}^{s}(X)=(1-\tilde{t}_i)((A_i - X)^T\otimes I_2) (M^s(\widetilde{Y}_i)-I_4)\int_{0}^{\tilde{t}_i}\frac{d}{dt} \emph{Vec} \left( \nabla \mathbf{ u}^s(Y_i(t,X)) \right) dt. 
\end{cases}
\end{equation}
\end{thm}
\begin{proof}
Without loss of generality, we only discuss the case $A_i\in T^+$ and $X\in T_{\ast}^-$. Following a procedure similar to that used in
\cite{2004LiLinLinRogers}, we have
\begin{equation}
\label{expan_2_1}
\begin{split}
\mathbf{ u}^+(A_i)=&\mathbf{ u}^-(X)+\int^{\tilde{t}_i}_0\frac{d}{dt}\mathbf{ u}^-(Y_i(t,X))dt+\int^{1}_{\tilde{t}_i}\frac{d}{dt}\mathbf{ u}^+(Y_i(t,X))dt\\
=&\mathbf{ u}^-(X)+ \nabla\mathbf{u}^-(X)(A_i-X)- \nabla\mathbf{u}^-(\widetilde{Y}_i)(A_i-\widetilde{Y}_i)+ \nabla\mathbf{u}^+(\widetilde{Y}_i)(A_i-\widetilde{Y}_i)\\
&+\int^{\tilde{t}_i}_0(1-t)\frac{d^2}{dt^2}\mathbf{ u}^-(Y_i(t,X))dt+\int^{1}_{\tilde{t}_i}(1-t)\frac{d^2}{dt^2}\mathbf{ u}^+(Y_i(t,X))dt,
\end{split}
\end{equation}
where the last two terms are actually $\mathbf{ R}^-_{i1}$ and $\mathbf{ R}^-_{i2}$. For the second and the third term on the right hand side of
\eqref{expan_2_1}, by applying \eqref{kro_vec}, we have
\begin{equation}
\begin{split}
\label{expan_2_2}
&\nabla\mathbf{ u}^-(X)(A_i-X)=((A_i-X)^T\otimes I_2) \textrm{Vec}(\nabla\mathbf{ u}^-(X)), \\
&\nabla\mathbf{ u}^-(\widetilde{Y}_i)(A_i-\widetilde{Y}_i)=((A_i-\widetilde{Y}_i)^T\otimes I_2) \textrm{Vec}(\nabla\mathbf{ u}^-(\widetilde{Y}_i)).
\end{split}
\end{equation}
For the fourth term on the right hand side of
\eqref{expan_2_1}, by applying \eqref{kro_vec} and Lemma \ref{MM_relat}, we have
\begin{equation}
\begin{split}
\label{expan_2_3}
\nabla\mathbf{ u}^+(\widetilde{Y}_i)(A_i-\widetilde{Y}_i)&=((A_i-\widetilde{Y}_i)^T\otimes I_2) \textrm{Vec}(\nabla\mathbf{ u}^+(\widetilde{Y}_i)) \\
&=(1-\tilde{t}_i)((A_i - X)^T\otimes I_2) M^-(\widetilde{Y}_i) \textrm{Vec}(\nabla\mathbf{ u}^-(\widetilde{Y}_i)).
\end{split}
\end{equation}
Moreover we note that
\begin{equation}
\begin{split}
\label{expan_2_4}
\nabla \mathbf{ u}^-(\widetilde{Y}_i) &= \int^{\tilde{t}_i}_0 \frac{d}{dt} \nabla \mathbf{ u}^-(Y_i(t,x)) dt + \nabla\mathbf{ u}^-(X).
\end{split}
\end{equation}
Finally, expansion \eqref{expan_2} follows from substituting \eqref{expan_2_4}, \eqref{expan_2_3} and \eqref{expan_2_2} into \eqref{expan_2_1}.
\end{proof}

For $X\in T_{\ast}$, we consider another group of expansion which only involves the first derivative of $\mathbf{ u}$.
\begin{thm}
\label{T_tild_exp}
On every interface element $T\in\mathcal{T}^i_h$, $\mathbf{ u}\in\mathbf{ PC}^2_{int}(T)$, for each $X\in T_{\ast}$, we have
 \begin{equation}
\label{T_tild_exp_eq}
\mathbf{ u}(A_i) = \mathbf{ u}(X) + \widetilde{\mathbf{R}}_i(X),  ~~~~ \emph{with} ~~\widetilde{\mathbf{R}}_i(X) = \int^1_0\frac{d}{dt}\mathbf{ u}(Y_i(t,X))dt.
\end{equation}
\end{thm}
\begin{proof}
The proof follows from a straightforward application of the same arguments used in \cite{2016GuoLinZhang} that only relies on the
continuity of $\mathbf{ u}$.
\end{proof}

We proceed to estimate the remainders in \eqref{expan_1_rem} and \eqref{expan_2_rem} in terms of the Hilbert norms associated with $\mathbf{ PH}^2_{int}(T)$. For every scalar function $u$, let $\nabla^2 u$ be its Hessian matrix. Then we note that
\begin{equation}
\label{Hessian_1}
\frac{d^2}{dt^2}\mathbf{ u}(Y_i(t,X)) =
\left[\begin{array}{c}
(A_i-X)^T\nabla^2u_1~(A_i-X) \\
(A_i-X)^T\nabla^2u_2~(A_i-X)
\end{array}\right],
\end{equation}
\begin{equation}
\label{Hessian_2}
\frac{d}{dt}\left(\nabla\mathbf{ u}(Y_i(t,X)) \right) =
\left[\begin{array}{c}
(A_i-X)^T\nabla^2u_1 \\
(A_i-X)^T\nabla^2u_2
\end{array}\right].
\end{equation}
Therefore we have
\begin{lemma}
\label{lem_rem_est_1}
Let $\mathbf{ u}\in\mathbf{ PC}^2_{int}(T)$, there exist constants $C>0$ independent of interface location such that
\begin{equation}
\begin{split}
\label{rem_est_1}
&\| \mathbf{ R}^s_i \|_{0,T_{\ast}^s}\leqslant Ch^2 |\mathbf{ u}|_{2,T}, ~~ i\in\mathcal{I}^s, ~s=\pm, \\
&\| \mathbf{ R}^s_{i1} \|_{0,T_{\ast}^s}\leqslant Ch^2 |\mathbf{ u}|_{2,T}, ~~ \| \mathbf{ R}^s_{i2} \|_{0,T_{\ast}^s}\leqslant Ch^2 |\mathbf{ u}|_{2,T}, ~~ i\in\mathcal{I}^{s'}, ~s=\pm.
\end{split}
\end{equation}
\end{lemma}
\begin{proof}
Let $\mathbf{ R}^s_i=(R^{1s}_i,R^{2s}_i)^T$, then according to \eqref{Hessian_1}, using Minkowski inequality and the fact $\|A_i-X \|\leqslant h$, we have
\begin{equation*}
\begin{split}
R^{js}_i(X)&= \left( \int_{T_{\ast}^s}\left( \int^1_0 (1-t)(A_i-X)^T\nabla^2 u^s_j(Y_i(t,X))(A_i-X)dt \right)^2 dX \right)^{\frac{1}{2}}\\
&\leqslant Ch^2 \int^1_0 \left( \int_{T_{\ast}^s}(1-t)^2\sum_{k,l=1}^2 | \partial_{x_kx_l} u^s_j |^2 \right)^{\frac{1}{2}}  dt \leqslant Ch^2 |u_j|_{2,T}, ~~~ j=1,2,
\end{split}
\end{equation*}
where we have used arguments similar to those used for the Lemma 4.1 in \cite{2016GuoLin}, and these estimates lead to
the first estimate in \eqref{rem_est_1}. The derivations for the estimates of $\mathbf{ R}^s_{i1}$ and $\mathbf{ R}^s_{i2}$ are similar.
\end{proof}

\begin{lemma}
\label{lem_rem_est_2}
Let $\mathbf{ u}\in\mathbf{ PC}^2_{int}(T)$, there exist constants $C>0$ independent of interface location such that
\begin{equation}
\label{rem_est_2}
\| \mathbf{ R}^s_{i3} \|_{0,T_{\ast}^s}\leqslant Ch^2 | \mathbf{ u} |_{2,T}, ~ i\in \mathcal{I}^{s'}, s=\pm.
\end{equation}
\end{lemma}
\begin{proof}
Let $\mathbf{ R}^s_{i3}=(R^{1s}_{i3},R^{2s}_{i3})^T$. Using \eqref{Hessian_2}, \eqref{M_bound}, the fact $\|A_i-X \|\leqslant h$ and $0\leqslant 1-\tilde{t}_i(X)\leqslant 1-t$, we have
$$
\| R^{js}_{i3} \|_{0,T_{\ast}^s}\leqslant Ch^2 \left( \int_{T^s_{\ast}} \left( \int^{\tilde{t}_i}_0 (1-t) \sum_{k,l=1}^2 \sum_{j=1}^2 |\partial_{ x_k  x_l} u_j| dt \right)^2 dX \right)^{\frac{1}{2}}.
$$
Then, applying the Minkowski inequality and Lemma 4.1 in \cite{2016GuoLin} to the inequality above yields
\begin{equation*}
\begin{split}
\| R^{js}_{i3} \|_{0,T_{\ast}^s} 
\leqslant Ch^2 \int^{\tilde{t}_i}_0 \left( \sum_{k,l=1}^2 \sum_{j=1}^2 \int_{T^s_{\ast}} (1-t)^2 |\partial_{x_k x_l} u_j|^2 dX \right)^{\frac{1}{2}} dt \leqslant Ch^2 (|u_1|_{2,T}+|u_2|_{2,T}),
\end{split}
\end{equation*}
from which \eqref{rem_est_2} readily follows.
\end{proof}

In addition, since $\mathbf{ u}\in \left[ H^{2}(T^s) \right]^2$, the Sobolev embedding theorem indicates $\mathbf{ u}\in \left[ W^{1,6}(T^s) \right]^2$, $s=\pm$. Therefore we can bound the remainder $\widetilde{\mathbf{ R}}_i$ in \eqref{T_tild_exp_eq} in terms of $W^{1,6}$-norm.

\begin{lemma}
\label{lem_T_tild_exp_est}
There exists a constant $C$ independent of the interface location such that when $h$ is small enough there holds
\begin{equation}
\label{T_tild_exp_est}
\| \widetilde{\mathbf{ R}}_i \|_{0,T_{\ast}} \leqslant Ch^2 \| \mathbf{ u} \|_{1,6,T}.
\end{equation}
\end{lemma}
\begin{proof}
We note that $T_{\ast} = \widetilde{T} \cup T_{int}$, and it is a small set
such that $|T_{\ast}|\leqslant Ch^3$ when the mesh is fine enough because
we know that $\abs{\widetilde{T}} \leq Ch^3$ \cite{2004LiLinLinRogers} and $|T_{\ast}|\leqslant Ch^3$ by Lemma \ref{T_int_area}.
Note that $\widetilde{\mathbf{ R}}_i=(\widetilde{R}^1_i,\widetilde{R}^2_i)^T$, and, by using \eqref{Hessian_2}, we have
$$
\widetilde{R}^j_i(X)=\int^1_0 \nabla u_j(Y_i(t,X))~(A_i-X)dt, ~~~ j=1,2.
$$
Then, applying arguments similar to those used for Lemma 3.2 in \cite{2016GuoLinZhang} and using the fact $|T_{\ast}|\leqslant Ch^3$, we have $\| \widetilde{R}^j_i \|_{0,T_{\ast}}\leqslant Ch^2 \|u_j\|_{1,6,T}$ for $j=1,2$ from which \eqref{T_tild_exp_est} follows.
\end{proof}


\section{Construction of IFE Spaces}
In this section, we construct local IFE spaces corresponding to their related finite element spaces $(T,\mathbf{ \Pi}_T,\mathbf{ \Sigma}_T)$ described in
Section \ref{sec:preliminaries}. As usual the local IFE space on every non-interface element $T$ is the standard vector polynomial space, i.e.,
\begin{equation}
\label{non_inter_ife_sp}
\mathbf{ S}_h(T)=\textrm{Span}\{ \bfpsi_{i,T}, \,\bfpsi_{i + \abs{\mathcal{I}},T}:\, i \in \mathcal{I} \},
\end{equation}
where $\bfpsi_{i,T}$ are given by \eqref{vec_bas}.
We note that a procedure to construct the local IFE spaces formed by piecewise linear polynomials on interface elements
was discussed in \cite{2010GongLi,2012LinZhang}, and a similar procedure was presented in \cite{2012LinZhang} for
the local IFE spaces formed by piecewise bilinear polynomials. However according to the example presented in \cite{2012LinZhang}, the linear system for determining a IFE shape function in these procedures can be singular in some cases. We now propose a new procedure so that the bilinear or the roated $Q_1$ IFE shape functions on every interface element can always be uniquely determined by the local degrees of freedom $\mathbf{ \Sigma}_T$.

\subsection{Local IFE Spaces}

Without loss of generality, we consider a typical interface element $T\in\mathcal{T}^i_h$ with $A_1=(0,0)^T$, $A_2=(h,0)^T$, $A_3=(0,h)^T$
when linear polynomials are discussed on a triangular $T$, $A_1=(0,0)^T$, $A_2=(h,0)^T$, $A_3=(0,h)^T$, $A_4=(h,h)^T$ for the bilinear case on
a rectangular $T$, and $A_1=(h/2,0)^T$, $A_2=(h,h/2)^T$, $A_3=(h/2,h)^T$, $A_4=(0,h/2)^T$ for the rotated $Q_1$ case on a a rectangular $T$.
According to \cite{2016GuoLin}, by considering rotation, there are only two possible cases of the interface configuration for the linear and bilinear cases, and 5 possible cases for the rotated $Q_1$ case, as illustrated in Figures \ref{fig:subfig_linear}-\ref{fig:subfig_nc}.
\begin{figure}[H]
\begin{minipage}[h]{0.5\textwidth}
\centering
\subfigure[Case 1]{
    \label{tri_case1} 
    \includegraphics[width=1.5in]{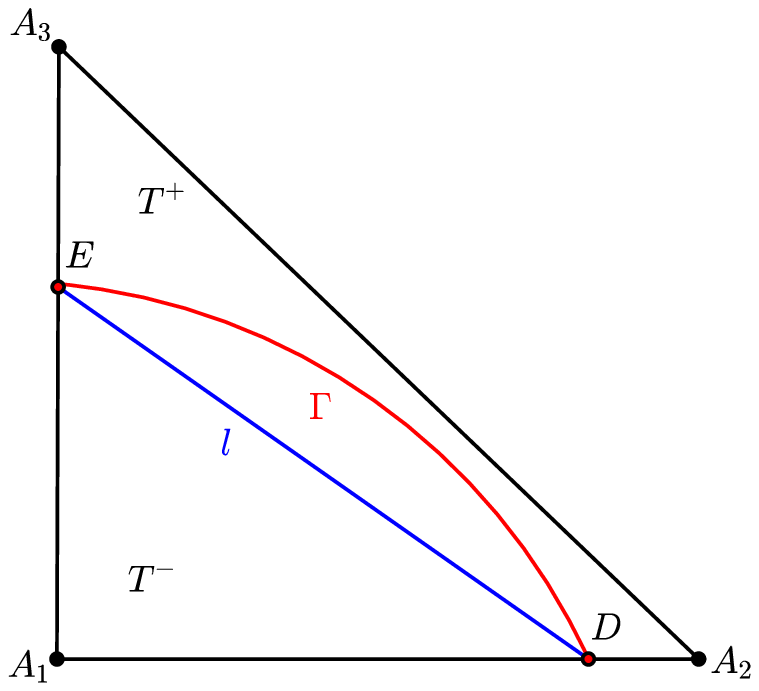}}
  \subfigure[Case 2]{
    \label{tri_case2} 
    \includegraphics[width=1.5in]{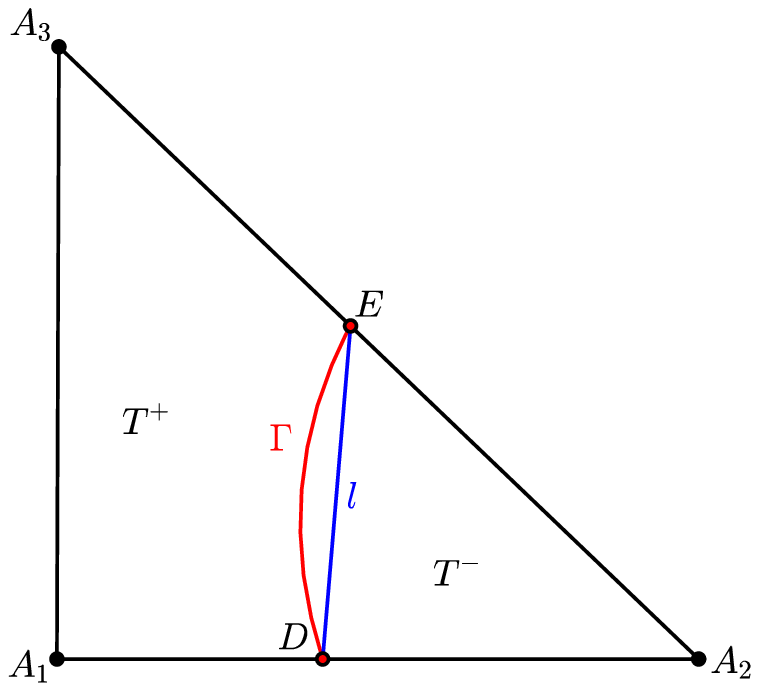}}
  \caption{Typical linear elements}
  \label{fig:subfig_linear} 
\end{minipage}
\begin{minipage}[h]{0.5\textwidth}
\centering
\subfigure[Case 1]{
    \label{case1} 
    \includegraphics[width=1.5in]{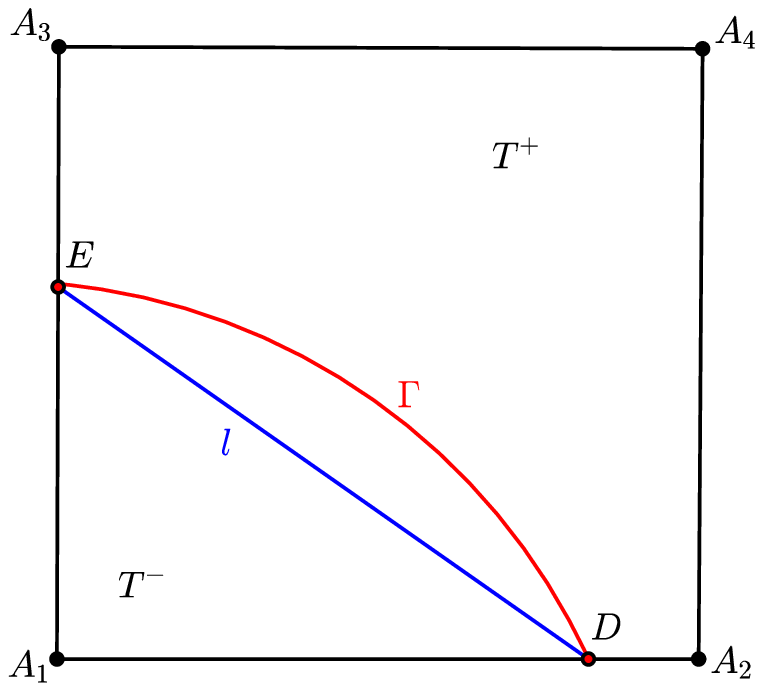}}
  \subfigure[Case 2]{
    \label{case2} 
    \includegraphics[width=1.5in]{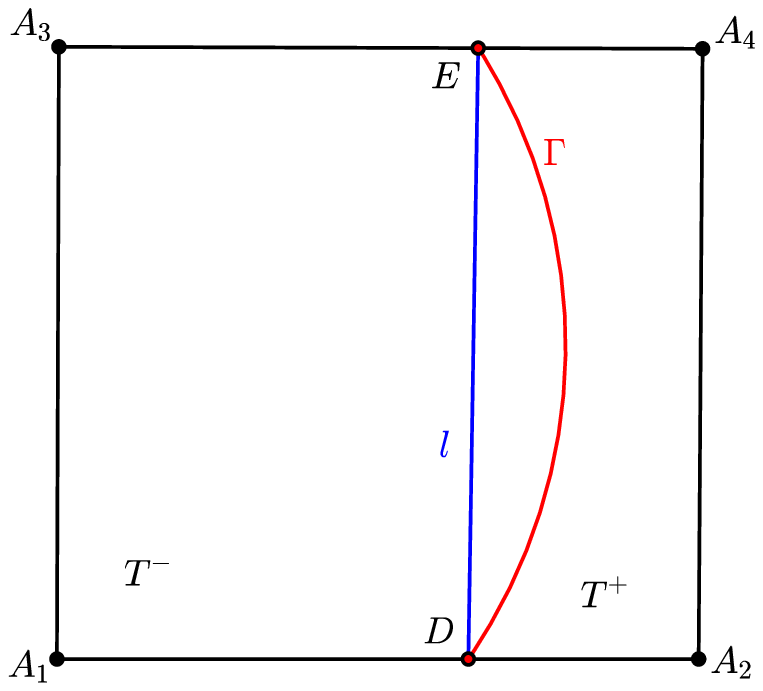}}
  \caption{Typical bilinear elements }
  \label{fig:subfig_bilinear} 
\end{minipage}
\end{figure}

\begin{figure}[H]
\centering
\subfigure[Case 1]{
    \label{nc_case1} 
    \includegraphics[width=1.2in]{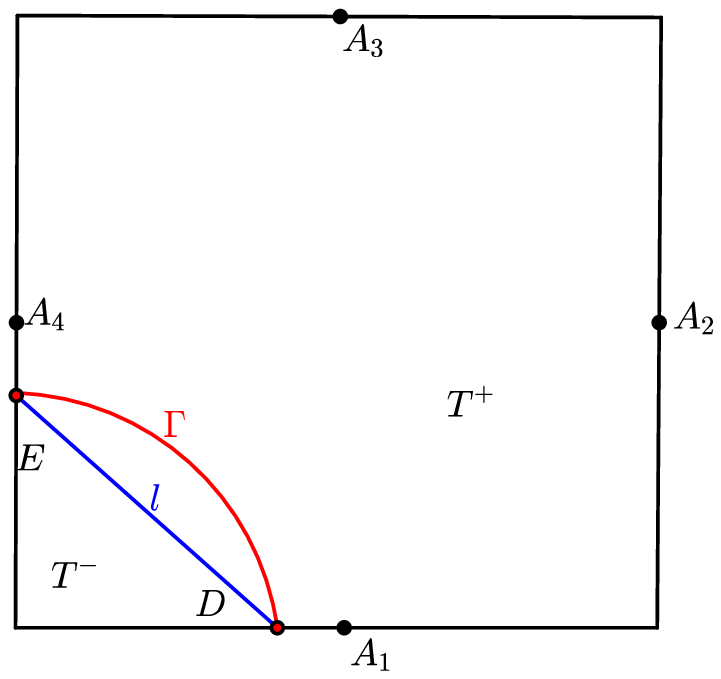}}
  \subfigure[Case 2]{
    \label{nc_case2} 
    \includegraphics[width=1.2in]{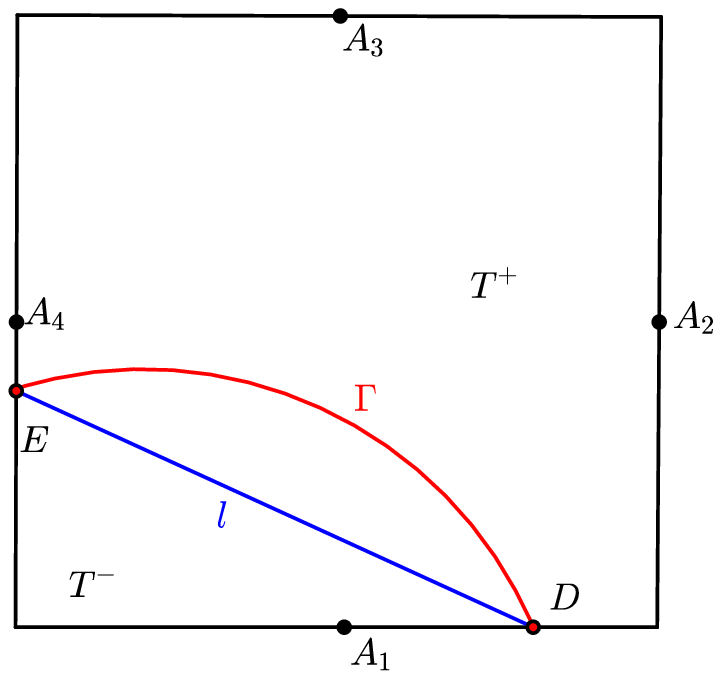}}
     \subfigure[Case 3]{
    \label{nc_case3} 
    \includegraphics[width=1.2in]{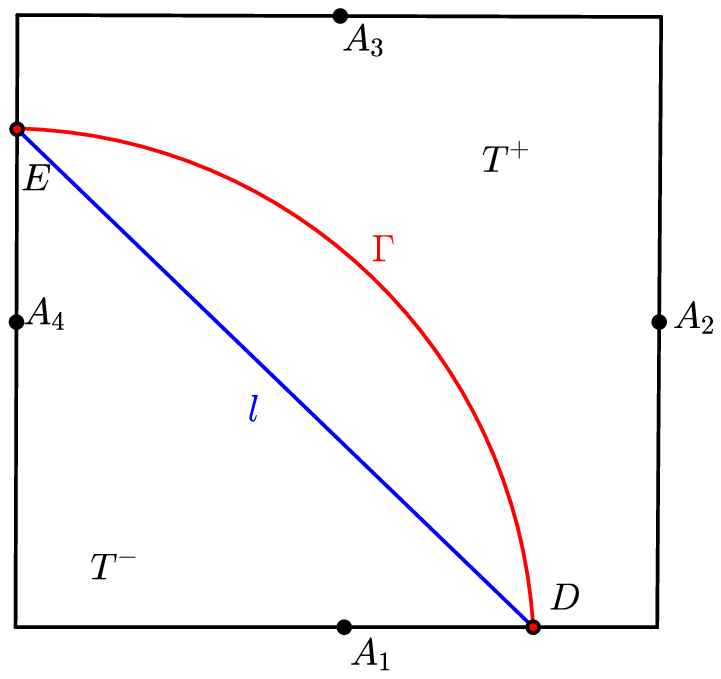}}
     \subfigure[Case 4]{
    \label{nc_case4} 
    \includegraphics[width=1.3in]{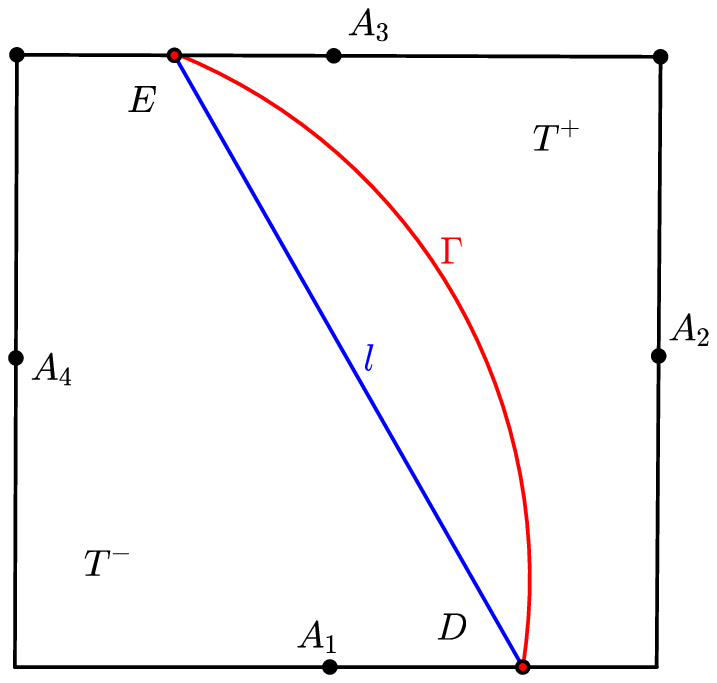}}
     \subfigure[Case 5]{
    \label{nc_case5} 
    \includegraphics[width=1.2in]{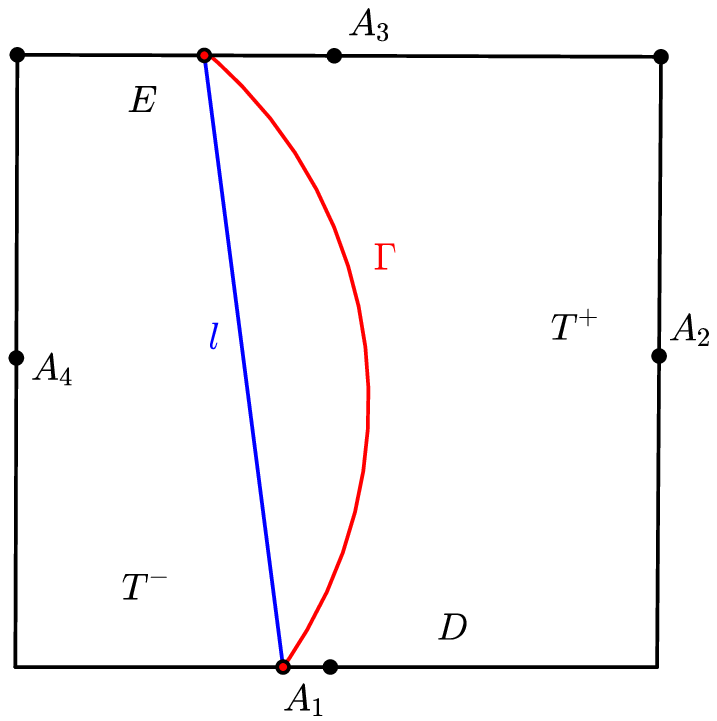}}
  \caption{Typical rotated $Q_1$ elements }
  \label{fig:subfig_nc} 
\end{figure}


\commentout{
In the following discussions, we let $\hat{\sigma}=\sigma^+-\sigma^-$ generated by the differences in Lam\'e parameters $\hat{\lambda}=\lambda^+-\lambda^-$, $\hat{\mu}=\mu^+-\mu^-$ {\color{red}($\hat{\sigma}$ is not well-defined???)}.
}

On an interface element $T$, we consider the elasticity IFE functions as piecewise vector polynomials in the following format:
\begin{equation}
\label{int_fun}
\bfphi_{T}(X)=
\left\{
\begin{aligned}
&\bfphi^-_T(X)\in \mathbf{ \Pi}_T, \;\;\; & \textrm{if} \;\;\; X \in \overline{T}^- , \\
&\bfphi^+_T(X)\in \mathbf{ \Pi}_T, \;\;\; & \textrm{if} \;\;\; X \in \overline{T}^+,
\end{aligned}
\right.
\end{equation}
with $\bfphi^+_T(X)$ and $\bfphi^-_T(X)$ satisfying that
\begin{eqnarray}
&&\begin{cases}
\bfphi^{-}_T|_{l}=\bfphi^{+}_T|_{l}, &\text{(for the linear case)}, \\
\bfphi^{-}_T|_{l}=\bfphi^{+}_T|_{l}, ~~d(\bfphi^{-}_T)=d(\bfphi^{+}_T), &\text{(for the bilinear/rotated $Q_1$ case)},
\end{cases} \label{lin_cond_1} \\
&&\sigma^+(\bfphi^+_T)(F)\;\bar{\mathbf{ n}}= \sigma^-(\bfphi^-_T)(F)\;\bar{\mathbf{ n}}, \label{lin_cond_2}
\end{eqnarray}
where $F$ is the point on $l$ which will be specified later and $d(\bfpsi)$ is a vector formed by the coefficients of the second degree term of
$\bfpsi \in \mathbf{ \Pi}_T$, i.e., the coefficient of $xy$ for a bilinear polynomial or the coefficient of $x^2-y^2$ for a rotated $Q_1$ polynomial.
Given a set of nodal-value vectors $\mathbf{ v}_i$, $i\in\mathcal{I}$, we further impose the nodal value condition:
\begin{equation}
\label{nod_val}
\bfphi_{T}(A_i)=\mathbf{ v}_i.
\end{equation}
Let $\Psi_{i,T}=[ \bfpsi_{i,T},\bfpsi_{i+|\mathcal{I}|,T} ]$, $i\in\mathcal{I}$ which is a 2-by-2 matrix basis function and let $L(X)=0$ be the equation of the line $l$ with $L(X)=\bar{\mathbf{ n}}\cdot(X-D)$. It is easy to see that $\Psi_{i,T}(A_j)=\delta_{i,j}I_2$, $i,j\in\mathcal{I}$.
Without loss of generality, we assume that $\abs{\mathcal{I}^+} \geq \abs{\mathcal{I}^-}$.
Then by \eqref{nod_val} and \eqref{lin_cond_1}, we can express \eqref{int_fun} as
\begin{equation}
\label{int_fun_2}
\bfphi_{T}(X)=
\left\{
\begin{aligned}
\bfphi^-_T(X) &=\bfphi^+_{T}(X)+ L(X)\mathbf{ c}_0 \;\;\; &\textrm{if} \;\;\; X \in \overline{T}^-,  \\
\bfphi^+_T(X) &=\sum_{i\in \mathcal{I}^+} \Psi_{i,T}(X) \mathbf{ v}_i+ \sum_{i\in \mathcal{I}^-} \Psi_{i,T}(X) \mathbf{ c}_i \;\;\; &\textrm{if} \;\;\; X \in \overline{T}^+,
\end{aligned}
\right.
\end{equation}
where $\mathbf{ c}_0=(c^1_0,c^2_0)^T$ and $\mathbf{ c}_i=(c^1_i,c^2_i)^T$, $i\in\mathcal{I}^-$ are to be determined.
Applying the jump condition for the stress tensor \eqref{lin_cond_2} to \eqref{int_fun_2}, we obtain
\begin{equation}
\label{sys_deri_1}
\begin{split}
\sigma^-\left( L \mathbf{ c}_0\right)(F)\bar{\mathbf{ n}}=&\hat{\sigma}(\bfphi^+_T)(F)\bar{\mathbf{ n}},
\end{split}
\end{equation}
where $\hat{\sigma}(\bfv)(X)$ for a vector function $\bfv$ is defined as follows:
\begin{equation}
\label{stress_tensor_hat}
\hat{\sigma}(\bfv)=(\hat{\sigma}_{ij}(\bfv))_{1\leqslant i,j \leqslant 2}, ~~\hat{\sigma}_{ij}(\mathbf{ v}) = \hat{\lambda} (\nabla\cdot \mathbf{ v}) \delta_{i,j} + 2\hat{\mu} \epsilon_{ij}(\mathbf{ v}),~~\text{with~~} \hat{\lambda}=\lambda^+-\lambda^-,~\hat{\mu}=\mu^+-\mu^-.
\end{equation}
Also, in $\hat{\sigma}(\bfphi^+_T)(X)$, the function $\bfphi^+_T$ is a polynomial so that it can be evaluated for any $X$, and this meaning
applies to similar situations from now on. By direct calculations, we have
\begin{equation}
\label{sys_deri_2}
\sigma^-(L \mathbf{ c}_0)(F)=
\left[\begin{array}{cc}
\bar{n}_1^2(\lambda^-+\mu^-)+\mu^- & \bar{n}_1\bar{n}_2(\lambda^-+\mu^-) \\
\bar{n}_1\bar{n}_2(\lambda^-+\mu^-) & \bar{n}_2^2(\lambda^-+\mu^-)+\mu^-
\end{array}\right]\left[\begin{array}{c} c^1_0 \\ c^2_0 \end{array}\right]:=K\mathbf{ c}_0.
\end{equation}
Then we note that
\begin{equation}
\label{P}
K=Q \mathcal{P}^- Q^T, ~~ \text{with} ~~\mathcal{P}^-=\left[\begin{array}{cc} (\lambda^- + 2\mu^-) & 0 \\ 0 & \mu^- \end{array}\right],~~ Q=[\bar{\mathbf{ n}},\bar{\mathbf{ t}}], 
\end{equation}
which is obviously non-singular. Hence $\mathbf{ c}_0$ is determined by
\begin{equation}
\label{c_0}
\mathbf{ c}_0 = K^{-1} \hat{\sigma}(\bfphi^+_T)(F)\bar{\mathbf{ n}}.
\end{equation}
Next we apply the nodal value condition \eqref{nod_val} for $j\in\mathcal{I}^-$ and \eqref{c_0} to \eqref{int_fun_2} to obtain
\begin{equation}
\begin{split}
\label{sys_deri_3}
K \mathbf{ c}_j + L(A_j) \sum_{i\in\mathcal{I}^-}\hat{\sigma}(\Psi_{i,T}\mathbf{ c}_i)(F)\bar{\mathbf{ n}}= K\mathbf{ v}_j-L(A_j) \sum_{i\in\mathcal{I}^+}\hat{\sigma}(\Psi_{i,T}\mathbf{ v}_i)(F)\bar{\mathbf{ n}},~~j\in\mathcal{I}^-.
\end{split}
\end{equation}
We now put equations in \eqref{sys_deri_3} into a matrix form. To this end, we first let $(j_1, j_2, \cdots, j_{\abs{\mathcal{I}}})$ be a permutation of $(1, 2, \cdots, \abs{\mathcal{I}})$ such that $j_k \in \mathcal{I}^-$ for $1 \leq k \leq \abs{\mathcal{I}^-}$ but $j_k \in \mathcal{I}^+$ for $\abs{\mathcal{I}^-} + 1 \leq k \leq \abs{\mathcal{I}}$. Consider three vectors $\bfc$, $\bfv^-$ and $\bfv^+$ such that
\begin{align*}
\mathbf{ c} = \Big[\mathbf{ c}_{j_k}\Big]_{k=1}^{\abs{\mathcal{I}^-}} \in \mathbb{R}^{2\abs{\mathcal{I}^-}}, ~~
\mathbf{ v}^- = \Big[\mathbf{ v}_{j_k}\Big]_{k=1}^{\abs{\mathcal{I}^-}} \in \mathbb{R}^{2\abs{\mathcal{I}^-}}, ~~
\mathbf{ v}^+ = \Big[\mathbf{ v}_{j_k}\Big]_{k=\abs{\mathcal{I}^-} + 1}^{\abs{\mathcal{I}}} \in \mathbb{R}^{2\abs{\mathcal{I}^+}}.
\end{align*}
\commentout{
\begin{align*}
\mathbf{ c}^T=\left[\mathbf{ c}_1^T,\cdots,\mathbf{ c}^T_{|\mathcal{I}^-|}\right], ~~(\mathbf{ v}^{-})^T = \left[\mathbf{ v}_1^T,\cdots,\mathbf{ v}^T_{|\mathcal{I}^-|}\right],~~(\mathbf{ v}^{+})^T = \left[\mathbf{ v}_{\abs{\mathcal{I}^-}+1}^T,\mathbf{ v}_{\abs{\mathcal{I}^-}+2}^T\cdots,\mathbf{ v}^T_{|\mathcal{I}|}\right].
\end{align*}
}
We adopt the following notations:
\begin{subequations}
\label{notation}
\begin{align}
& \overline{K} =I_{|\mathcal{I}^-|}\otimes K \in\mathbb{R}^{2|\mathcal{I}^-|\times2|\mathcal{I}^-|}, ~~\overline{L} = \Big[ L(A_{j_k}) I_2 \Big]_{k = 1}^{\abs{\mathcal{I}^-}} \in \mathbb{R}^{2|\mathcal{I}^-|\times2}, \label{Lbar} \\
& \overline{\Psi}^- = \Big[ \overline{\Psi}_{j_k} \Big]_{k = 1}^{\abs{\mathcal{I}^-}} \in \mathbb{R}^{2|\mathcal{I}^-|\times2}, ~~
  \overline{\Psi}^+ = \Big[ \overline{\Psi}_{j_k} \Big]_{k = \abs{\mathcal{I}^-} + 1}^{\abs{\mathcal{I}}} \in \mathbb{R}^{2|\mathcal{I}^+|\times2}, \label{Psibar_1} \\
\textrm{with}~ & \overline{\Psi}_j = \left[\begin{array}{cc}
\hat{\sigma}(\bfpsi_{j,T})(F)\bar{\mathbf{ n}} & \hat{\sigma}(\bfpsi_{j+|\mathcal{I}|,T})(F)\bar{\mathbf{ n}}
\end{array}\right]^T\in \mathbb{R}^{2\times2} ~ \label{Psibar_2},~~1 \leq j \leq \abs{\mathcal{I}}.
\end{align}
\end{subequations}
For any vector $\mathbf{ r}\in\mathbb{R}^{2\times1}$, we note the identity $\hat{\sigma}(\Psi_{i,T}\mathbf{ r})(F)\bar{\mathbf{ n}}=\overline{\Psi}^T_i\mathbf{ r}$.
Hence by using the matrices defined in \eqref{Lbar}-\eqref{Psibar_2}, we can represent equations in \eqref{sys_deri_3} as follows:
\begin{align}
    &(\overline{K} + \overline{L} ~ \overline{\Psi}^{-T} )\mathbf{ c}=\mathbf{ b},   \label{lin_sys}\\
    \text{with}~~~~~~ &\mathbf{ b} = \overline{K}\mathbf{ v}^- - \overline{L} ~\overline{\Psi}^{+T}\mathbf{ v}^+. \label{b}
\end{align}
We note that the coefficient matrix in \eqref{lin_sys} is a generalized {\em Sherman-Morrison} matrix formed by matrices $\overline{K}, \overline{L}$ and
$\overline{\Psi}$.

\begin{rem}
\label{rem_unisol}
When $F=(D+E)/2$, the linear and bilinear IFE shape functions given by \eqref{int_fun_2} with coefficients determined by \eqref{int_fun_2} and \eqref{lin_sys}
coincide with those in \cite{2012LinZhang}.
\end{rem}

For linear IFE functions, because \eqref{lin_cond_2} is independent of the choice of $F$, the new construction procedure proposed above is the same as
the one considered in \cite{2010GongLi,2012LinZhang} and the unisolvance of a linear IFE function can only be conditionally guaranteed \cite{2012LinZhang}. However, with the proposed new construction procedure, the unisolvance for the bilinear and the rotated $Q_1$ IFE functions can always be guaranteed with
a suitable choice for $F$, and we proceed to discuss this important feature.

\commentout{
However under such a choice of $F$, the linear system \eqref{lin_sys} is singular for some interface configuration and special Lam\'e parameters \cite{2012LinZhang} for both the linear and bilinear cases. In our formulation, we provide another choice of $F$ such that \eqref{lin_sys} is always non-singular for the bilinear case as well as the rotated $Q_1$ case regardless of the interface location and Lam\'e parameters. However since the gradients of $\bfpsi_i$, $i\in\mathcal{I}$, are constants for the linear case, the location of $F$ does not have any affect on the \eqref{lin_sys}. In such a case, we refer readers to \cite{2012LinZhang} for the conditions on Lam\'e parameters such that the linear IFE shape functions are unisolvent.
}

First, for the rotated $Q_1$ IFE functions in Case 1 as illustrated in Figure \ref{fig:subfig_nc}(a), we note that there is no $\bfc_i, ~i \in \mathcal{I}^-$ coefficients in the formulation \eqref{int_fun_2} and $\mathbf{ c}_0$ is uniquely determined by \eqref{c_0}, and this means that the unisolvence for this case is always guaranteed.

To discuss other cases, we define two parameters $d$ and $e$ for describing the interface-element intersection points $D$ and $E$ for those typical rectangular interface elements illustrated in Figures \ref{fig:subfig_bilinear}-\ref{fig:subfig_nc}:

\begin{itemize}
\item
We let $d = \norm{D-A_1}/h, ~e = \norm{E-A_1}/h$ for {\bf Case 1} in Figure \ref{fig:subfig_bilinear} and
{\bf Case 2} and {\bf Case 3} in Figure \ref{fig:subfig_nc}.

\item
We let $d = \norm{D-A_1}/h, ~e = \norm{E-A_3}/h$ for {\bf Case 2} in Figure \ref{fig:subfig_bilinear} and
{\bf Case 4} and {\bf Case 5} in Figure \ref{fig:subfig_nc}.
\end{itemize}

We start from some estimates for the following two auxiliary functions:
\begin{equation}
\label{gnt}
g_n(X)=\sum_{i\in\mathcal{I}^-}L(A_i)\nabla\psi_{i,T}(X)\cdot\bar{\mathbf{ n}},  ~~~~ g_t(X)=\sum_{i\in\mathcal{I}^-}L(A_i)\nabla\psi_{i,T}(X)\cdot\bar{\mathbf{ t}}.
\end{equation}

\begin{lemma}
\label{uni_lem_1}
On each rectangular interface element $T\in\mathcal{T}^i_h$, let $F_0=t_0D+(1-t_0)E$ such that
\begin{itemize}
\item
when it is a bilinear element in \textbf{Case 1} illustrated in Figure \ref{fig:subfig_bilinear}, assume $t_0=e/(d+e)$;

\item
when it is a bilinear element in \textbf{Case 2} illustrated in Figure \ref{fig:subfig_bilinear}, assume $t_0=1-e$
if $d\geqslant e$, $t_0=1-d$ if $e > d$;

\item
when it is a rotate $Q_1$ element in \textbf{Case 2} or \textbf{Case 3} illustrated in Figure \ref{fig:subfig_nc}, assume $t_0=1$
when $d\geqslant e$ or $t_0=0$ when $e> d$

\item
when it is a rotate $Q_1$ element in \textbf{Case 4} or \textbf{Case 5} illustrated in Figure \ref{fig:subfig_nc}, assume $t_0=1/2$.
\end{itemize}
\commentout{
\begin{itemize}
\item $d=\| D - A_1 \|/h$, $e=\| E - A_1 \|/h$ for \textbf{Case 1} in Figure \ref{fig:subfig_bilinear}, $t_0=e/(d+e)$;
\item $d=\| D - A_1 \|/h$, $e=\| E - A_3 \|/h$ for \textbf{Case 2} in Figure \ref{fig:subfig_bilinear}, $t_0=1-e$ if $d\geqslant e$, $t_0=1-d$ if $e\geqslant d$;
\end{itemize}
and for the rotated $Q_1$ elements, assume
\begin{itemize}
\item $d=\| D - A_1 \|/h$, $e=\| E - A_1 \|/h$ for the \textbf{Case 1} and \textbf{Case 2} in Figure \ref{fig:subfig_nc}, $t_0=1$ if $d\geqslant e$, $t_0=0$ if $e\geqslant d$;
\item $d=\| D - A_1 \|/h$, $e=\| E - A_3 \|/h$ for the \textbf{Case 4} and \textbf{Case 5} in Figure \ref{fig:subfig_nc}, $t_0=1/2$;
\end{itemize}
}
Then
\begin{subequations}
\begin{align}
& (1-g_n(F_0))^2-g^2_t(F_0)\geqslant 0, ~~~ g_n^2(F_0) - g_t^2(F_0)\geqslant 0, \label{gtn_est} \\
& g_n(F_0)\in [0,1], ~~~ g_t(F_0)\in [-1,1].  \label{gn_est}
\end{align}
\end{subequations}
\end{lemma}
\begin{proof}
We only provide a proof for the \textbf{Case 1} of the bilinear elements and similar arguments can be applied to other cases. By direct calculation, we can verify that
\begin{equation*}
\begin{split}
&g_n^2(F_0) - g_t^2(F_0)=\frac{4d^3e^3(d+e-de)^2}{(d^2 + e^2)^2(d+e)^2}\geqslant0;\\
&(1-g_n(F_0))^2-g^2_t(F_0) =\frac{-d^2 e^2 (e^2(1-d)-d^2(1-e))^2 + (e^2(1 - d) + d^2 (1 - e + e^2))^2(d+e)^2}{(d^2 + e^2)^2(d+e)^2}\geqslant 0,
\end{split}
\end{equation*}
which lead to \eqref{gtn_est}.
For \eqref{gn_est}, the first inequality is just a special case of Lemma 5.1 in \cite{2016GuoLin} and the second inequality is a consequence of \eqref{gtn_est}.
\end{proof}

\begin{lemma}
\label{uni_lem_2}
The matrix in the linear system \eqref{lin_sys} is non-singular if and only if the follwoing matrix is non-singular:
\begin{equation}
\label{Xi_mat}
\Xi(F)=\mathcal{P}^- +
\left[\begin{array}{cc}
(\hat{\lambda}+2\hat{\mu})g_n(F) & \hat{\lambda}g_t(F) \\
\hat{\mu}g_t(F) & \hat{\mu}g_n(F)
\end{array}\right].
\end{equation}
\end{lemma}
\begin{proof}
Note that the matrix in \eqref{lin_sys} is in a generalized \textit{Sherman-Morrison} format. Since $\overline{K}$ is invertible, the linear system \eqref{lin_sys} is non-singular if and only if the matrix
\begin{equation}
\label{uni_lem_2_1}
I_2+\overline{\Psi}^{-T}\overline{K}^{-1}\overline{L}=I_2 + \sum_{j\in\mathcal{I}^-}L(A_j)\overline{\Psi}^T_jK^{-1}= (K + \sum_{j\in\mathcal{I}^-}L(A_j)\overline{\Psi}_j )K^{-1}
\end{equation}
is invertible. Then by using \eqref{P}, we can directly verify that
\begin{equation}
\label{uni_lem_2_2}
Q(K + \sum_{j\in\mathcal{I}^-}L(A_j)\overline{\Psi}_j )Q^T = P^- +
\left[\begin{array}{cc}
(\hat{\lambda}+2\hat{\mu})g_n(F) & \hat{\lambda}g_t(F) \\
\hat{\mu}g_t(F) & \hat{\mu}g_n(F)
\end{array}\right]
\end{equation}
which leads to the conclusion of this lemma because $Q$ is invertible.
\end{proof}

\begin{lemma}
\label{uni_lem_3}
With the $F_0$ specified in Lemma \ref{uni_lem_1}, we have
\begin{equation}
\label{det_Xi}
\textrm{Det}(\Xi(F_0))>2 \left(\min\{ \mu^+, \mu^- \}\right)^2.
\end{equation}
\end{lemma}
\begin{proof}
By direct calculations according to \eqref{Xi_mat}, we have
\begin{equation*}
\begin{split}
\label{Xi}
\textrm{Det}(\Xi(F_0))=&\lambda^+\mu^+(g_n^2-g_t^2)+\lambda^-\mu^-((1-g_n)^2-g_t^2)\\
&+\lambda^-\mu^+((1-g_n)g_n+g_t^2)+\lambda^+\mu^-(g_n(1-g_n)+g_t^2)\\
&+2(\mu^+)^2g_n^2+2(\mu^-)^2(1-g_n)^2+4\mu^+\mu^-(1-g_n)g_n,
\end{split}
\end{equation*}
in which $g_n = g_n(F_0)$ and $g_t = g_t(F_0)$. Then, apply estimates in Lemma \ref{uni_lem_1} to the above, we have
$
\textrm{Det}(\Xi(F_0))\geqslant 2(\mu^+)^2g_n^2+2(\mu^-)^2(1-g_n)^2 > 2 \left(\min\{ \mu^+, \mu^- \}\right)^2.
$
\end{proof}
Finally we can prove the following main theorem in this section.

\begin{thm}[Unisolvence]
\label{unisolvence}
Let $T\in\mathcal{T}^i_h$ be a rectangular interface element with $F=F_0$ specified in Lemma \ref{uni_lem_1}. Then given any vector $\mathbf{ v}\in\mathbb{R}^{2|\mathcal{I}^-|\times1}$, for the bilinear and rotated $Q_1$ elements, there exists one and only one IFE shape function satisfying \eqref{int_fun}-\eqref{nod_val}.
\end{thm}
\begin{proof}
The proof is directly based on Lemma \ref{uni_lem_2} and Lemma \ref{uni_lem_3}.
\end{proof}

\begin{rem}
According to the generalized \textit{Sherman-Morrison} formula and \eqref{uni_lem_2_1}, \eqref{uni_lem_2_2}, we can give an analytical formula for the coefficients $\mathbf{ c}$ in \eqref{lin_sys} as
\begin{equation}
\begin{split}
\label{SheMo}
\mathbf{ c}=&\overline{K}^{-1}\mathbf{ b} - \overline{K}^{-1}\overline{L}(I_2 + \overline{\Psi}^{-T}\overline{K}^{-1}\overline{L})^{-1}\overline{\Psi}^{-T}\overline{K}^{-1}\mathbf{ b}\\
=&\overline{K}^{-1}\mathbf{ b} - \overline{K}^{-1}\overline{L}KQ^T\Xi^{-1}Q\overline{\Psi}^{-T}\overline{K}^{-1}\mathbf{ b}.
\end{split}
\end{equation}
Here it is important to note that $\overline{K}$ is a diagonal block matrix formed only by the 2-by-2 matrix $K$ and $\Xi$ is also a 2-by-2 matrix so that their inverses are easy to calculate analytically. Hence, if preferred, there is no need to solve for $\bfc$ numerically because of \eqref{SheMo}.
\end{rem}

By taking the nodal value vector $\mathbf{ v}$ to be unit vectors, we construct the IFE shape functions satisfying the weak jump conditions \eqref{lin_cond_1}-\eqref{lin_cond_2} and
\begin{equation}
\label{ife_vec_bas}
\bfphi_{i,T}(A_j)=
\left\{\begin{array}{cc}
\delta_{i,j},  \\
0,
\end{array}\right.
~i=1,\cdots,|\mathcal{I}|,
~~~~ \textrm{and} ~~~~
\bfphi_{i,T}(A_j)=
\left\{\begin{array}{cc}
0,  \\
\delta_{i-|\mathcal{I}|,j},
\end{array}\right.
~ i=|\mathcal{I}|+1,\cdots,2|\mathcal{I}|.
\end{equation}
The local IFE spaces on interface elements $T\in\mathcal{T}^i_h$ is then defined as
\begin{equation}
\label{inter_ife_sp}
\mathbf{ S}_h(T)=\textrm{Span}\{ \bfphi_{i,T}, \,\bfphi_{i+\abs{\mathcal{I}},T}:\, i \in \mathcal{I} \}.
\end{equation}
The local IFE spaces defined by \eqref{non_inter_ife_sp} and \eqref{inter_ife_sp} can be used to construct an IFE space over the whole domain
$\Omega$ according to the need of a finite element scheme. For example, by enforcing the continuity at the mesh nodes, we can consider the following global IFE space:
\begin{equation}
\begin{split}
\label{glob_sp}
\mathbf{ S}_h(\Omega) = &\left\{ \mathbf{ v}\in [L^2(\Omega)]^2 \,:\, \mathbf{ v}|_T\in\mathbf{ S}_h(T); \right. \\
  & \left. ~ \mathbf{ v}|_{T_1}(N)=\mathbf{ v}|_{T_2}(N)~ \forall N\in \mathcal{N}_h, ~ \forall T_1, T_2 \in \mathcal{T}_h ~\textrm{such~that}~ N\in T_1\cap T_2  \right\}.
\end{split}
\end{equation}


\subsection{Properties of IFE Shape Functions}
In this subsection, we discuss some fundamental properties of the proposed IFE shape functions.
We tacitly assume that, on each interface element $T \in \mathcal{T}_h^i$, $\bfphi_{i,T}, 1 \leq i \leq 2|\mathcal{I}|$ are the bilinear or the rotated $Q_1$
IFE shape functions constructed according to Theorem \ref{unisolvence} or they are the linear IFE shape functions constructed
under the conditions specified by Theorem 4.7 in \cite{2012LinZhang}.

\begin{thm}[Boundedness]
\label{thm_bound}
There exists a constant $C$ such that the following estimates are valid for IFE shape functions on each interface element:
\begin{equation}
\label{bound_1}
| \bfphi_{i,T} |_{k,\infty,T} \leqslant Ch^{-k}, ~~ k=0,1,2, ~~1 \leq i \leq 2|\mathcal{I}|,~~\forall T \in \mathcal{T}_h^i.
\end{equation}
\end{thm}
\begin{proof}
For the bilinear or the rotated $Q_1$ IFE shape functions, we note that \eqref{P} yields $\| \overline{K} \|\leqslant C$. And \eqref{Lbar} shows $\| \overline{L} \|\leqslant Ch$ because $| L |_{0,\infty,T}\leqslant Ch$, and $\| \overline{\Psi}^s \| \leqslant Ch^{-1}$. So we have $\| \mathbf{ b} \|\leqslant C$, of which the constants $C$ only depends on Lam\'e parameters. Next \eqref{gn_est} and \eqref{det_Xi} suggest $\| \Xi^{-1} \|\leqslant C$. So by the formula \eqref{SheMo}, we have $\|\mathbf{ c} \|\leqslant C$ and then use \eqref{c_0} to show $\| \mathbf{ c}_0 \| \leqslant Ch^{-1}$. Therefore $\| \mathbf{ c}_0 L \|_{0,\infty,T}\leqslant C$ and $| \mathbf{ c}_0 L |_{1,\infty,T}\leqslant Ch^{-1}$ because $| L |_{0,\infty,T}\leqslant Ch$ and $| L |_{1,\infty,T}\leqslant C$. In addition, it is easy to see $| \mathbf{ c}_0 L |_{2,\infty,T}=0$, since $L$ is a linear function. Finally, applying these estimates and \eqref{fe_bound} to \eqref{int_fun_2} leads to
\eqref{bound_1}. Similar arguments can be used for the linear IFE shape functions.
\end{proof}

For simplicity of presentation, we denote the following matrix shape functions:
\begin{equation}
\label{basis_mat}
\Phi_{i,T}(X) = \left[ \bfphi_{i,T}(X), ~\bfphi_{i+|\mathcal{I}|,T}(X) \right], ~~ i\in\mathcal{I}.
\end{equation}

\begin{thm}[Partition of Unity]
\label{POU}
On each interface element $T\in\mathcal{T}^i_h$, we have
\begin{equation}
\label{POU_2}
\sum_{i\in\mathcal{I}} \Phi_{i,T}(X) = I_2, ~~~ \sum_{i\in\mathcal{I}} \partial_{x_j}\Phi_{i,T}(X) = \mathbf{ 0}_{2\times2}, ~~~ \sum_{i\in\mathcal{I}} \partial_{x_jx_k}\Phi_{i,T}(X) = \mathbf{ 0}_{2\times2}, ~~j, k=1,2.
\end{equation}
\end{thm}
\begin{proof}
By direct verifications we can see that vector functions $\bfphi^1=(1,0)^T$ and $\bfphi^2=(0,1)^T$ satisfy the weak jump conditions \eqref{lin_cond_1} and \eqref{lin_cond_2} exactly; hence, they are in the IFE space $\mathbf{ S}_h(T)$. Then the unisolvence of the IFE function leads to
the first identity in \eqref{POU_2}. And the second and third identity in \eqref{POU_2} are just the derivatives of the first one.
\end{proof}

\begin{rem}
The first identity in \eqref{POU_2} was proved in \cite{2012LinZhang} for the linear and bilinear IFE shape functions by direct verifications.
\end{rem}


For the proposed IFE shape functions, we consider the following $2$-by-$4$ matrix functions:
\begin{subequations}
\label{identity}
\begin{equation}
\label{identity_m}
\Lambda_{-}(X) = \sum_{i\in\mathcal{I}}\left( (A_i-X)^T\otimes( \Phi^{-}_{i,T}(X)) \right)+ \sum_{i\in\mathcal{I}^+}\left( (A_i-\overline{X}_i)^T\otimes(\Phi^{-}_{i,T}(X)) \right)  (\overline{M}^- -I_4),
\end{equation}
\begin{equation}
\label{identity_p}
\Lambda_{+}(X) = \sum_{i\in\mathcal{I}}\left( (A_i-X)^T \otimes(\Phi^{+}_{i,T}(X)) \right) + \sum_{i\in\mathcal{I}^-} \left( (A_i-\overline{X}_i)^T\otimes( \Phi^{+}_{i,T}(X)) \right) (\overline{M}^+ -I_4),
\end{equation}
\end{subequations}
where $\overline{X}_i$, $i\in\mathcal{I}$ are arbitrary points on $l$, and we further use them to define
\begin{equation}
\label{identity_rela}
\Lambda^+(X)=\Lambda_+(X), \;\;\; \Lambda^-(X)=\Lambda_-(X)\overline{M}^+.
\end{equation}
First we show that both $\Lambda^-(X)$ and $\Lambda^+(X)$ are well defined i.e., they are independent of $\overline{X}_i\in l$, $i\in\mathcal{I}$.

\begin{lemma}
\label{well_define}
The matrix functions $\Lambda^-(X)$ and $\Lambda^+(X)$ are independent with the points $\overline{X}_i\in l$, $i\in\mathcal{I}$.
\end{lemma}
\begin{proof}
Let $\overline{X}_i$ and $\overline{X}'_i$ be two points on $l$. Then, by Lemma \ref{eigen_vec}, we have
\begin{equation}
\begin{split}
&\left( (A_i-\overline{X}_i)^T\otimes( \Phi^{s}_i(X)) \right) (\overline{M}^s -I_4) - \left( (A_i-\overline{X}'_i)^T\otimes( \Phi^{s}_i(X)) \right) (\overline{M}^s -I_4) \\
=&\| \overline{X}_i - \overline{X}'_i \| \left[ \Phi^{s}_i, \Phi^{s}_i \right] ~ \left[\alpha_1,\alpha_2\right] (\overline{M}^s -I_4) = \mathbf{ 0}, ~~~ s=\pm,
\end{split}
\end{equation}
where, as in Lemma \ref{eigen_vec}, $\alpha_1=(-\bar{n}_2,0,\bar{n}_1,0)^T$ and $\alpha_2=(0,-\bar{n}_2,0,\bar{n}_1)^T$. Hence,
by \eqref{identity_p}-\eqref{identity_rela}, functions $\Lambda^-(X)$ and $\Lambda^+(X)$ are independent of $\overline{X}_i$, $i\in\mathcal{I}$.
\end{proof}

Lemma \ref{well_define} allows us to consolidate $\overline{X}_i$, $i\in\mathcal{I}$ in $\Lambda^-(X)$ and $\Lambda^+(X)$ into a single point $\overline{X}\in l$. Then, by using \eqref{POU_2}, we rewrite these two functions as follows:
\begin{subequations}
\label{identity}
\begin{equation}
\label{identity_mm}
\Lambda^{-}(X) = \sum_{i\in\mathcal{I}^-} (A_i-\overline{X})^T \otimes \Phi^-_i(X) \overline{M}^{+} + \sum_{i\in\mathcal{I}^+} (A_i-\overline{X})^T\otimes \Phi^-_i(X) - \left( (X-\overline{X})^T\otimes I_2 \right) \overline{M}^+.
\end{equation}
\begin{equation}
\label{identity_pp}
\Lambda^{+}(X) = \sum_{i\in\mathcal{I}^-} (A_i-\overline{X})^T \otimes \Phi^+_i(X) \overline{M}^{+} + \sum_{i\in\mathcal{I}^+} (A_i-\overline{X})^T\otimes \Phi^+_i(X) -  (X-\overline{X})^T\otimes I_2 ,
\end{equation}
\end{subequations}
For every fixed $\overline{X}$, we consider the following piecewise 2-by-4 matrix function:
\begin{equation}
\label{V}
V(X) =
\left\{\begin{array}{cc}
(X-\overline{X})^T\otimes I_2 & \text{if} \; X\in \overline{T}^+ , \\
\left( (X-\overline{X})^T\otimes I_2 \right) \overline{M}^+& \text{if} \; X\in \overline{T}^- .
\end{array}\right.
\end{equation}

\begin{lemma}
\label{V}
On every interface element $T \in \mathcal{T}_h^i$, each column of $V(X)$ is in the local IFE space $\mathbf{ S}_h(T)$.
\end{lemma}
\begin{proof}
Clearly, each column of $V(X)$ restricted on either $\overline{T}^+$ or $\overline{T}^-$ is in the corresponding polynomial space $\mathbf{ \Pi}_T$. Furthermore, we note that $V^-(\overline{X})=V^+(\overline{X})$ and
$$
\left[\begin{array}{c} \partial_{x_1} V^+ \\ \partial_{x_2} V^+ \end{array}\right] = I_4=\overline{M}^+\overline{M}^- = \left[\begin{array}{c} \partial_{x_1} V^- \\ \partial_{x_2} V^- \end{array}\right] \overline{M}^-,
$$
which, together with the fact $d(V)=0_{2\times4}$, shows each column of $V$ satisfies \eqref{lin_cond_1} and \eqref{lin_cond_2} simultaneously; thus, it is in the corresponding IFE space.
\end{proof}

\begin{thm}
\label{funda_identity}
For every interface element $T\in \mathcal{T}_h$, we have the identities $\Lambda_+=0_{2\times4}$ and $\Lambda_-=0_{2\times4}$. And let $I^1=[I_2,0_{2\times2}]$, $I^2=[0_{2\times2},I_2]$, for $j, k=1,2$, $s=\pm$, we also have
\begin{subequations}
\label{identity_d}
\begin{equation}
\label{identity_d1}
\sum_{i\in\mathcal{I}}\left( (A_i-X)^T \otimes( \partial_{x_j}\Phi^{s}_i(X)) \right) + \sum_{i\in\mathcal{I}^{s'}} \left( (A_i-\overline{X}_i)^T\otimes( \partial_{x_j}\Phi^{s}_i(X)) \right) (\overline{M}^s -I_4) = I^j,
\end{equation}
\begin{equation}
\label{identity_d2}
\sum_{i\in\mathcal{I}}\left( (A_i-X)^T\otimes( \partial_{x_jx_k}\Phi^{s}_i(X)) \right)+ \sum_{i\in\mathcal{I}^{s'}}\left( (A_i-\overline{X}_i)^T\otimes( \partial_{x_jx_k}\Phi^{s}_i(X)) \right)  (\overline{M}^s -I_4) = 0_{2\times4}.
\end{equation}
\end{subequations}
\end{thm}
\begin{proof}
We construct a piecewise defined $2$-by-$4$ matrix function:
\begin{equation}
\label{lambda}
\Lambda(X) =
\left\{\begin{array}{cc}
\Lambda^-(X) & \text{if} \; X\in \overline{T}^-, \\
\Lambda^+(X) & \text{if} \; X\in \overline{T}^+.

\end{array}\right.
\end{equation}
According to \eqref{identity_pp} and \eqref{identity_mm}, the first two terms in $\Lambda^-(X)$ and $\Lambda^+(X)$ are the linear combination of the IFE shape functions with the same coefficients. Lemma \ref{V} shows that their last terms together form a function in the local IFE space. So each column of the piecewise defined matrix function \eqref{lambda} belongs to $\mathbf{ S}_h(T)$ for every interface element $T$. In addition, by \eqref{identity_p} and \eqref{identity_m}, we can see that $\Lambda(A_i)=0_{2\times4}$, $i\in\mathcal{I}$. Hence, by the unisolvence of IFE functions, we know that each column of $\Lambda(X)$ must be $0_{2\times1}$ which leads to
$\Lambda_{\pm}(X) = 0_{2\times4}$ because $\overline{M}^+$ is non-singular. Furthermore, \eqref{identity_d1} and \eqref{identity_d2} can be obtained by differentiating \eqref{identity_p} and \eqref{identity_m}.
\end{proof}


\subsection{Interpolation Error Analysis}
In this subsection, we use the results above to show the optimal approximation capabilities of the proposed IFE spaces by estimating the errors of the Lagrange type interpolation operators. Again, we assume the conditions for the unisolvence of IFE shape functions are satisfied, i.e., $F=F_0$ given in Lemma \ref{uni_lem_1} in the construction of the bilinear and rotated $Q_1$ IFE shape functions and the conditions in Theorem 4.7 in \cite{2012LinZhang} are satisfied for the linear case.
To study the approximation capabilities, we consider the following local Lagrange type interpolation operator: $I_{h,T}:\mathbf{ C}^0(T)\rightarrow \mathbf{ S}_h(T)$ with
\begin{equation}
\label{loc_interp}
I_{h,T}\mathbf{ u}(X)= \begin{cases}
\sum_{i\in\mathcal{I}}\Psi_{i,T}(X)\mathbf{ u}(A_i), & \text{if~} T \in \mathcal{T}_h^n, \\
\sum_{i\in\mathcal{I}}\Phi_{i,T}(X)\mathbf{ u}(A_i), & \text{if~} T \in \mathcal{T}_h^i;
\end{cases} ~~\forall \bfu \in \mathbf{ C}^0(T).
\end{equation}
The global interpolation operator $I_h$ defined on $\mathbf{ C}^0(\Omega)$ can defined correspondingly such that
\begin{equation}
\label{glob_interp}
(I_h \mathbf{ u})|_{T}=I_{h,T} \mathbf{ u}, \;\;\; \forall T\in \mathcal{T}_h,~~\forall \bfu \in \mathbf{ C}^0(\Omega).
\end{equation}
Applying the standard scaling argument \cite{1994BrennerScott,1978Ciarlet,1992RanacherTurek} onto each component of the vector function
$\mathbf{ u}=(u_1,u_2)^T \in \mathbf{H}^2(T)$, we can show that for all the non-interface elements, there holds
\begin{equation}
\label{non_int_interp_est}
\| I_{h,T} u_i - u_i \|_{0,T}+h| I_{h,T} u_i - u_i |_{1,T} + h^2 | I_{h,T} u_i - u_i |_{2,T} \leqslant Ch^2|u_i|_{2,T}, ~~ i=1,2,~~ \forall ~T \in \mathcal{T}_h^n.
\end{equation}
However on interface elements $T\in\mathcal{T}^i_h$, due to the jump conditions \eqref{jump_contin} and \eqref{jump_stress}, the two components of $\mathbf{ u}$ have to be treated together. To estimate the errors, we give the following two theorems on the expansion of the interpolation operator on $T\in\mathcal{T}^i_h$.

\begin{thm}
\label{interpolation_exp}
On each interface element $T\in\mathcal{T}^i_h$, assume $\mathbf{ u}\in\mathbf{PC}^2_{int}(T)$, then, for any $\overline{X}_i\in l$,
the following expansions hold for every $X \in T_{\ast}^s$:
\begin{subequations}\label{interpolation_exp_1}
\begin{equation}
\label{interpolation_exp_1_1}
I_{h,T}\mathbf{ u}(X)-\mathbf{ u}(X) = \sum_{i\in\mathcal{I}^{s'}} \Phi_{i,T}(X) (\mathbf{ E}_i^s(X) + \mathbf{ F}_i^s(X))+\sum_{i\in\mathcal{I}} \Phi_{i,T}(X) \mathbf{ R}_i^s(X), ~s = \pm,
\end{equation}
\begin{equation}
\label{interpolation_exp_1_2}
\partial_{x_j}( I_{h,T} \mathbf{ u}(X) - \mathbf{ u}(X) )=\sum_{i\in\mathcal{I}^{s'}} \partial_{x_j}\Phi_{i,T}(X)(\mathbf{ E}_i^s(X) + \mathbf{ F}_i^s(X)) + \sum_{i\in\mathcal{I}} \partial_{x_j}\Phi_{i,T}(X)\mathbf{ R}_i^s(X), ~s = \pm,
\end{equation}
\begin{equation}
\label{interpolation_exp_1_3}
\partial_{x_jx_k} I_{h,T} \mathbf{ u}(X)=\sum_{i\in\mathcal{I}^{s'}} \partial_{x_jx_k}\Phi_{i,T}(X)(\mathbf{ E}_i^s(X) + \mathbf{ F}_i^s(X)) + \sum_{i\in\mathcal{I}} \partial_{x_jx_k}\Phi_{i,T}(X)\mathbf{ R}_i^s(X), ~s = \pm,
\end{equation}
\end{subequations}
where $j, k = 1,2$, $\mathbf{ R}^s_i(X)$ are given in \eqref{expan_1_rem}, \eqref{expan_2_rem} and
\begin{equation}
\begin{split}
\label{EF}
&\mathbf{E}^s_i(X)= \left( (A_i - \widetilde{Y}_i)^T\otimes I_2 \right)\left( M^s(\widetilde{Y}_i)-\overline{M}^s \right)\emph{Vec}(\nabla\mathbf{ u}^s(X)),  \\
&\mathbf{F}^s_i(X)= -\left( (\widetilde{Y}_i-\overline{X}_i)^T\otimes I_2 \right)\left(\overline{M}^s - I_4 \right)\emph{Vec}(\nabla\mathbf{ u}^s(X)).
\end{split}
\end{equation}
\end{thm}
\begin{proof}
The argument is similar to \cite{2016GuoLin} by applying the fundamental identity Theorem \ref{funda_identity} onto the interpolation operator \eqref{loc_interp}. Expanding the nodal values $\mathbf{ u}(A_i)$, $i\in\mathcal{I}$, about $X\in T_{\ast}^s$ in the interpolation operator \eqref{loc_interp} by \eqref{expan_1} and \eqref{expan_2}, we obtain
\begin{equation}
\begin{split}
\label{interpolation_exp_eq_1}
I_{h,T}\mathbf{ u}(X) 
&= \sum_{i\in\mathcal{I}}\Phi_{i,T}(X) \mathbf{ u}(X) + \left( \sum_{i\in\mathcal{I}} \Phi_{i,T}(X) \left((A_i-X)^T\otimes I_2\right) \right)\textrm{Vec}(\nabla\mathbf{ u}^s(X)) \\
&+ \left( \sum_{i\in\mathcal{I}^{s'}}\Phi_{i,T}(X) \left((A_i-\widetilde{Y}_i)^T\otimes I_2\right)(M^s - I_4) \right)\textrm{Vec}(\nabla\mathbf{ u}^s(X)) + \sum_{i\in\mathcal{I}}\Phi_{i,T}(X)\mathbf{ R}^s_i.
\end{split}
\end{equation}
Note that for any vector $\mathbf{ r}\in\mathbb{R}^{2\times1}$, there holds
\begin{equation}
\label{interpolation_exp_eq_2}
\Phi_{i,T}(X) \left(\mathbf{ r}^T\otimes I_2\right)= \mathbf{ r}^T\otimes \Phi_{i,T}(X).
\end{equation}
Then we apply Theorem \ref{funda_identity} onto the second term in \eqref{interpolation_exp_eq_1} to have
\begin{equation}
\begin{split}
\label{interpolation_exp_eq_3}
I_{h,T}\mathbf{ u}(X) &= \sum_{i\in\mathcal{I}}\Phi_{i,T}(X) \mathbf{ u}(X) - \left( \sum_{i\in\mathcal{I}^{s'}}\left((A_i - \overline{X}_i)^T\otimes \Phi_{i,T}(X) \right)(\overline{M}^{s}-I_4)  \right)\textrm{Vec}(\nabla\mathbf{ u}^s(X)) \\
&+ \left( \sum_{i\in\mathcal{I}^{s'}} \left((A_i-\widetilde{Y}_i)^T\otimes \Phi_{i,T}(X) \right)(M^s - I_4) \right)\textrm{Vec}(\nabla\mathbf{ u}^s(X)) + \sum_{i\in\mathcal{I}}\Phi_{i,T}(X)\mathbf{ R}^s_i.
\end{split}
\end{equation}
Then, \eqref{interpolation_exp_1_1} follows by applying partition of unity, the fact $A_i-\overline{X}_i=(A_i-\widetilde{Y}_i)+(\widetilde{Y}_i-\overline{X}_i)$, and the identity \eqref{interpolation_exp_eq_2} to
\eqref{interpolation_exp_eq_3}. For \eqref{interpolation_exp_1_2}, we apply \eqref{expan_1} and \eqref{expan_2} to $\partial_{x_j}I_{h,T}\mathbf{ u}(X)=\sum_{i\in\mathcal{I}}\partial_{x_j}\Phi_{i,T}(X)\mathbf{ u}(A_i)$ to obtain
\begin{equation*}
\begin{split}
\label{interpolation_exp_eq_4}
\partial_{x_j} I_{h,T}\mathbf{ u}(X)
&= \sum_{i\in\mathcal{I}}\partial_{x_j}\Phi_{i,T}(X) \mathbf{ u}(X) + \left( \sum_{i\in\mathcal{I}} \partial_{x_j}\Phi_{i,T}(X) \left((A_i-X)^T\otimes I_2\right) \right)\textrm{Vec}(\nabla\mathbf{ u}^s(X)) \\
&+ \left( \sum_{i\in\mathcal{I}^{s'}} \partial_{x_j}\Phi_{i,T}(X) \left((A_i-\widetilde{Y}_i)^T\otimes I_2\right)(M^s - I_4) \right)\textrm{Vec}(\nabla\mathbf{ u}^s(X)) + \sum_{i\in\mathcal{I}} \partial_{x_j}\Phi_{i,T}(X)\mathbf{ R}^s_i.
\end{split}
\end{equation*}
By using \eqref{POU_2}, \eqref{identity_d1} and \eqref{interpolation_exp_eq_2} in the above, we have
\begin{equation*}
\begin{split}
\label{interpolation_exp_eq_5}
\partial_{x_j}I_{h,T}\mathbf{ u}(X) &= I^j \textrm{Vec}(\nabla\mathbf{ u}^s(X))  - \left( \sum_{i\in\mathcal{I}^{s'}}\left((A_i - \overline{X}_i)^T\otimes \partial_{x_j}\Phi_{i,T}(X) \right)(\overline{M}^{s}-I_4)  \right)\textrm{Vec}(\nabla\mathbf{ u}^s(X)) \\
&+ \left( \sum_{i\in\mathcal{I}^{s'}} \left((A_i-\widetilde{Y}_i)^T\otimes \partial_{x_j}\Phi_{i,T}(X)\right)(M^s - I_4) \right)\textrm{Vec}(\nabla\mathbf{ u}^s(X)) + \sum_{i\in\mathcal{I}} \partial_{x_j}\Phi_{i,T}(X)\mathbf{ R}^s_i,
\end{split}
\end{equation*}
which is in the same format as \eqref{interpolation_exp_eq_3} because $I^j \textrm{Vec}(\nabla\mathbf{ u}^s(X)=\partial_{x_j}\mathbf{ u}(X)$.
Therefore, \eqref{interpolation_exp_1_2} follows from arguments used to derive \eqref{interpolation_exp_1_1} from \eqref{interpolation_exp_eq_3}.
Finally, \eqref{interpolation_exp_1_3} can be derived very similarly by applying \eqref{expan_1} and \eqref{expan_2} in $\partial_{x_jx_k}I_{h,T}\mathbf{ u}(X)=\sum_{i\in\mathcal{I}}\partial_{x_jx_k}\Phi_{i,T}(X)\mathbf{ u}(A_i)$ and then using \eqref{POU_2}, \eqref{identity_d2} and \eqref{interpolation_exp_eq_2}.
\end{proof}

\begin{rem}
\label{rem_interpolation_exp_1_3}
We note that \eqref{interpolation_exp_1_3} is trivial for the linear IFE shape functions $\Phi_{i,T}$ since each side is simply a zero vector. And the non-trivial one is the bilinear case with $j=1$, $k=2$ and the rotated-$Q_1$ case with $j=k=1$ or $j = k = 2$.
\end{rem}

In addition, for $X\in T_{\ast}$, we consider a simpler expansion as the following.
\begin{thm}
\label{interpolation_exp_T_tild}
On each interface element $T\in\mathcal{T}^i_h$, assume $\mathbf{ u}\in\mathbf{PC}^2_{int}(T)$, the following expansions hold for every $X \in T_{\ast}$:
\begin{subequations}\label{interpolation_exp_1}
\begin{equation}
\label{interpolation_exp_T_tild_1}
I_{h,T}\mathbf{ u}(X)-\mathbf{ u}(X) = \sum_{i\in\mathcal{I}} \Phi_{i,T}(X)\widetilde{ \mathbf{ R}}_i(X),
\end{equation}
\begin{equation}
\label{interpolation_exp_T_tild_2}
\partial_{x_j}( I_{h,T} \mathbf{ u}(X) - \mathbf{ u}(X))= -\partial_{x_j}\mathbf{ u}(X)  + \sum_{i\in\mathcal{I}} \partial_{x_j}\Phi_{i,T}(X)\widetilde{\mathbf{ R}}_i(X),
\end{equation}
\begin{equation}
\label{interpolation_exp_T_tild_3}
\partial_{x_jx_k}( I_{h,T} \mathbf{ u}(X) - \mathbf{ u}(X) )= -\partial_{x_jx_k}\mathbf{ u}(X)
+ \sum_{i\in\mathcal{I}} \partial_{x_jx_k}\Phi_{i,T}(X)\widetilde{\mathbf{ R}}_i(X),
\end{equation}
\end{subequations}
where $j, k = 1,2$ and $\widetilde{\mathbf{ R}}_i$ is given in \eqref{T_tild_exp_eq}.
\end{thm}
\begin{proof}
They can be directly verified by applying \eqref{T_tild_exp_eq} to the interpolation operator \eqref{loc_interp}
\end{proof}

Now we are ready to derive estimates for the interpolation error.
\begin{thm}
\label{T_bar_interp_est}
There exists a constant $C$ independent of the interface location such that the following estimate holds for every $\mathbf{ u}\in \mathbf{ PH}^2_{int}(T)$:
\begin{equation}
\label{T_bar_interp_est_eq}
\| I_{h,T}\mathbf{ u} - \mathbf{ u} \|_{0,T_{\ast}^s} + h| I_{h,T}\mathbf{ u} - \mathbf{ u} |_{1,T_{\ast}^s} + h^2| I_{h,T}\mathbf{ u} - \mathbf{ u} |_{2,T_{\ast}^s}\leqslant C h^2 (|\mathbf{ u}|_{1,T}+|\mathbf{ u}|_{2,T}), ~s=\pm, ~\forall T \in \mathcal{T}_h^i.
\end{equation}
\end{thm}
\begin{proof}
First by \eqref{M_aprox} and $\|A_i-\widetilde{Y}_i\|\leqslant Ch$, we have
$
\| \mathbf{ E}^s_i \|_{0,T^s_{\ast}} \leqslant Ch^2 |\mathbf{ u}|_{1,T}
$, $i\in\mathcal{I}$, $s=\pm$.
Noticing $\|\widetilde{Y}_i-\overline{X}_i\|\leqslant Ch^2$ from Lemma 3.2 in \cite{2016GuoLin}, we have
$
\| \mathbf{ F}^s_i \|_{0,T^s_{\ast}} \leqslant Ch^2 |\mathbf{ u}|_{1,T}
$, $i\in\mathcal{I}$, $s=\pm$. Now putting these estimates, Lemmas \ref{lem_rem_est_1}, \ref{lem_rem_est_2} and Theorem \ref{thm_bound} into \eqref{interpolation_exp_1_1} and \eqref{interpolation_exp_1_2}, for $s=\pm$, $j=1,2$, we have
\begin{align*}
    & \| I_{h,T}\mathbf{ u}-\mathbf{ u} \|_{0,T_{\ast}^s} \leqslant \sum_{i\in\mathcal{I}^{s'}} C ( \|\mathbf{ E}_i^s\|_{0,T_{\ast}^s} + \|\mathbf{ F}_i^s\|_{0,T_{\ast}^s} )+\sum_{i\in\mathcal{I}} C \|\mathbf{ R}_i^s\|_{0,T_{\ast}^s}\leqslant Ch^2(|\mathbf{ u}|_{1,T}+|\mathbf{ u}|_{2,T}),  \\
    &  \|\partial_{x_j} I_{h,T} \mathbf{ u} - \partial_{x_j}\mathbf{ u} \|_{0,T_{\ast}^s} \leqslant \sum_{i\in\mathcal{I}^{s'}} Ch^{-1}(\|\mathbf{ E}_i^s\|_{0,T_{\ast}^s} + \|\mathbf{ F}_i^s\|_{0,T_{\ast}^s}) + \sum_{i\in\mathcal{I}} Ch^{-1} \|\mathbf{ R}_i^s\|_{0,T_{\ast}^s}\leqslant Ch(|\mathbf{ u}|_{1,T}+|\mathbf{ u}|_{2,T}).
    \end{align*}
In addition, by \eqref{interpolation_exp_1_3}, for $j, k = 1,2$, we have
\begin{align*}
  \|\partial_{x_jx_k} I_{h,T} \mathbf{ u} - \partial_{x_jx_k}\mathbf{ u} \|_{0,T_{\ast}^s} &\leqslant \| \partial_{x_jx_k}\mathbf{ u} \|_{0,T_{\ast}^s} +  \sum_{i\in\mathcal{I}^{s'}} Ch^{-2}(\|\mathbf{ E}_i^s\|_{0,T_{\ast}^s} + \|\mathbf{ F}_i^s\|_{0,T_{\ast}^s}) + \sum_{i\in\mathcal{I}} Ch^{-2} \|\mathbf{ R}_i^s\|_{0,T_{\ast}^s}\\
  &\leqslant C(|\mathbf{ u}|_{1,T}+|\mathbf{ u}|_{2,T}).
\end{align*}
These estimates lead to the desired result for $\mathbf{ u}\in\mathbf{ PC}^2_{int}(T)$. Then the estimation for $\mathbf{ u}\in\mathbf{ PH}^2_{int}(T)$ can be obtained from the density Hypothesis \textbf{(H4)}.
\end{proof}


\begin{thm}
\label{T_tild_interp_est}
There exists a constant $C$ independent of the interface location such that, when the mesh is fine enough, the following estimate holds for every $\mathbf{ u}\in \mathbf{ PH}^2_{int}(T)$:
\begin{equation}
\label{T_tild_interp_est_eq}
\| I_{h,T}\mathbf{ u} - \mathbf{ u} \|_{0,T_{\ast}} + h| I_{h,T}\mathbf{ u} - \mathbf{ u} |_{1,T_{\ast}} + h^2| I_{h,T}\mathbf{ u} - \mathbf{ u} |_{2,T_{\ast}}\leqslant C h^2 (\| \mathbf{ u} \|_{1,6,T}+\| \mathbf{ u} \|_{2,T}),~~\forall T \in \mathcal{T}_h^i.
\end{equation}
\end{thm}
\begin{proof}
By the same arguments used for the prove of Lemma \ref{lem_T_tild_exp_est}, we know that $|T_{\ast}|\leqslant Ch^3$ for a mesh fine enough. Then,
according to \eqref{fe_bound}, Lemma \ref{lem_T_tild_exp_est} and Theorem \ref{interpolation_exp_T_tild}, for $j, k = 1,2$, we have
$$
\| \Phi_{i,T}\widetilde{\mathbf{ R}}_i \|_{0,T_{\ast}} \leqslant Ch^2 \|\mathbf{ u}\|_{1,6,T},  ~~~ \| \partial_{x_j}\Phi_{i,T}\widetilde{\mathbf{ R}}_i \|_{0,T_{\ast}} \leqslant Ch\|\mathbf{ u}\|_{1,6,T} ~~~ \textrm{and} ~~~ \| \partial_{x_jx_k}\Phi_{i,T}\widetilde{\mathbf{ R}}_i \|_{0,T_{\ast}} \leqslant C\|\mathbf{ u}\|_{1,6,T}.
$$
Besides, the H\"older's inequality implies
$$
\left( \int_{T_{\ast}} (\partial_{x_j}u_m)^2 dX \right)^{\frac{1}{2}} \leqslant \left( \int_{T_{\ast}} 1^{\frac{3}{2}} dX\right)^{\frac{1}{3}} \left( \int_{T_{\ast}} (\partial_{x_j}u_m)^6 dX \right)^{\frac{1}{6}} \leqslant Ch \| u_m \|_{1,6,T},~m = 1, 2,
$$
where we use the fact $|T_{\ast}|\leqslant Ch^3$. For the second derivatives, it is easy to see that
$$
\left( \int_{T_{\ast}} (\partial_{x_jx_k}u_j)^2 dX \right)^{\frac{1}{2}} \leqslant C \| u_j \|_{2,T},~~j, k = 1, 2.
$$
By applying the estimates above to \eqref{interpolation_exp_T_tild_1}, \eqref{interpolation_exp_T_tild_2} and \eqref{interpolation_exp_T_tild_3}, we have \eqref{T_tild_interp_est_eq} for all $\mathbf{ u}\in\mathbf{ PC}^2_{int}(T)$. Again the result for $\mathbf{ u}\in\mathbf{ PH}^2_{int}(T)$ follows from the density Hypothesis \textbf{(H4)}.
\end{proof}

Finally by combing the results above, we can prove the optimal approximation capabilities for the proposed IFE space through the following error estimation for the global interpolation operator.

\begin{thm}
There exists a constant $C$ independent of the interface location such that, when the mesh $\mathcal{T}_h$ is fine enough, the following estimate holds for every $\mathbf{ u}\in \mathbf{ PH}^2_{int}(T)$:
\begin{equation}
\label{glob_interp_est}
\| I_h \mathbf{ u} - \mathbf{ u} \|_{0,\Omega} + h|I_h\mathbf{ u} - \mathbf{ u}|_{1,h,\Omega} + h^2|I_h\mathbf{ u} - \mathbf{ u}|_{2,h,\Omega} \leqslant Ch^2 \| \mathbf{ u} \|_{2,\Omega},
\end{equation}
where $\abs{\,\cdot\,}_{1,h,\Omega}$ and $\abs{\,\cdot\,}_{2,h,\Omega}$ are the usual discrete semi-norms defined according to the mesh $\mathcal{T}_h$.
\end{thm}
\begin{proof}
By putting \eqref{T_bar_interp_est_eq} and \eqref{T_tild_interp_est_eq} together over all the elements $T$, we have
$$
\| I_h \mathbf{ u} - \mathbf{ u} \|_{0,\Omega} + h|I_h\mathbf{ u} - \mathbf{ u}|_{1,h,\Omega} + h^2|I_h\mathbf{ u} - \mathbf{ u}|_{2,h,\Omega} \leqslant Ch^2 (\| \mathbf{ u} \|_{2,\Omega} + \| \mathbf{ u} \|_{1,6,\Omega}).
$$
Then using the inequality $\|w\|^2_{1,p,\Omega}\leqslant C\|w\|^2_{2,\Omega}$ for any $w\in W^{1,p}(\Omega)$ from \cite{RenWei1994}, we have \eqref{glob_interp_est}.
\end{proof}

\section{Numerical Examples}
In this section we demonstrate the optimal approximation capabilities of the IFE spaces by numerical examples. We use an example similar to that given in \cite{2012LinZhang} in which the solution domain is $\Omega=[-1,1]\times[-1,1]$ and the exact solution $\mathbf{ u}$ to the elasticity interface problem
described by \eqref{elas_eq0}-\eqref{jump_stress} is
\begin{equation}
\label{true_solu}
\mathbf{u}(x_1,x_2)=
\left[\begin{array}{c} u_1(x_1,x_2) \\ u_2(x_1,x_2) \end{array}\right]=
\begin{cases}
\left[\begin{array}{c} u^-_1(x_1,x_2) \\ u^-_2(x_1,x_2) \end{array}\right] =
\left[\begin{array}{c} \frac{a^2b^2}{\lambda^-}r^{\alpha_1} \\ \frac{a^2b^2}{\lambda^-}r^{\alpha_2} \end{array}\right] & \text{if} \; X\in \Omega^- , \\
\left[\begin{array}{c} u^+_1(x_1,x_2) \\ u^+_2(x_1,x_2) \end{array}\right] =
\left[\begin{array}{c} \frac{a^2b^2}{\lambda^+}r^{\alpha_1}+\left( \frac{1}{\lambda^-}-\frac{1}{\lambda^+}a^2b^2 \right] \\  \frac{a^2b^2}{\lambda^+}r^{\alpha_2}+\left( \frac{1}{\lambda^-}-\frac{1}{\lambda^+}a^2b^2 \right) \end{array}\right] & \text{if} \; X\in \Omega^+ ,
\end{cases}
\end{equation}
where $\lambda^-=1$, $\lambda^+=5$, $\mu^-=2$ and $\mu^+=10$, $a=b=\pi/6.28$, $\alpha_1=5$, $\alpha_2=7$ and $r(x_1,x_2)=x_1^2/a^2+x_2^2/b^2$, the interface $\Gamma$ is a circle defined by the zero level set $r(x_1,x_2)-1=0$ and $\Omega^-=\{ (x_1,x_2)^T~:~r(x_1,x_2)<1\}$, $\Omega^+=\{ (x_1,x_2)^T~:~r(x_1,x_2)>1\}$. All of the numerical results presented below are generated by the proposed bilinear IFE space on Cartesian meshes. The errors are measured in both the $L^2$ and semi-$H^1$ norms over a sequence a meshes with the size specified by $h$.

We first present the numerical results for the interpolation operator $I_h\mathbf{ u}$ defined by \eqref{loc_interp} and \eqref{glob_interp} in Table \ref{table:bilinearIterpolationError}. The convergence rate $r$ is estimated from the errors computed on two consecutive meshes. As expected, the numerical results clear show that the interpolation errors converge optimally.

Next, for the IFE solution to the elasticity interface problem, we consider an IFE Galerkin scheme discussed in \cite{2010GongLi,2012LinZhang}: find $\mathbf{ u}_h\in\mathbf{ S}_h(\Omega)$ such that $\mathbf{ u}_h=I_h\mathbf{ g}$ on $\partial\Omega$,
\begin{equation}
\label{IFE_scheme}
a(\mathbf{ u}_h,\mathbf{ v}_h) = L(v_h), ~~ \forall \mathbf{ v}_h\in \mathbf{ S}_{h,0}(\Omega),
\end{equation}
where $\mathbf{ S}_{h,0}=\{\mathbf{ v}_h\in\mathbf{ S}_h~:~ \mathbf{ v}|_{\partial\Omega}=\mathbf{ 0}\}$ and
\begin{equation}
\label{aL}
a(\mathbf{ u}_h,\mathbf{ v}_h) = \sum_{T\in\mathcal{T}_h} \int_T 2\mu \epsilon(\mathbf{ u}_h):\epsilon(\mathbf{ v}_h) + \lambda \textrm{div}(\mathbf{ u}_h)\textrm{div}(\mathbf{ v}_h) dX, ~~ \textrm{and} ~~ L(\mathbf{ v}_h) =\sum_{T\in\mathcal{T}_h} \int_{T} \mathbf{ f}\cdot\mathbf{ v}_h dX.
\end{equation}
Errors of the IFE solution are listed in Table \ref{table:bilinearSolutionError} in which, again, we use the errors generated from two consecutive meshes to estimate the convergence rate. The data in this table clear demonstrate that the IFE solutions $\mathbf{ u}_h$ also converges to the exact solution $\mathbf{ u}$ optimally.

\begin{table}[H]
\begin{center}
\begin{tabular}{|c |c c|c c|}
\hline
$h$    & $\|\mathbf{ u} - I_h\mathbf{ u}\|_{0,\Omega}$ & rate   & $|\mathbf{ u} - I_h\mathbf{ u}|_{1,\Omega}$ & rate   \\ \hline
1/10  & 5.6990E-1                  &             & 6.8680E+0    &                   \\ \hline
1/20  & 1.4528E-1                 &  1.9719 & 3.4933E+0             & 0.9753        \\ \hline
1/40   & 3.6502E-2                 & 1.9928 & 1.7544E+0             & 0.9936 \\ \hline
1/80   & 9.1372E-3                 & 1.9981 & 8.7822E-1               & 0.9984 \\ \hline
1/160  & 2.2851E-3                 & 1.9995 & 4.3924E-1               & 0.9996 \\ \hline
1/320  & 5.7132E-4                 & 1.9999 & 2.1964E-1               & 0.9999 \\ \hline
1/640  & 1.4283E-4                 & 2.0000 & 1.0982E-1               & 1.0000 \\ \hline
1/1280 & 3.5709E-5                 & 2.0000 & 5.4911E-2               & 1.0000 \\ \hline
\end{tabular}
\end{center}
\caption{IFE Interpolation errors and rates for the bilinear IFE functions}
\label{table:bilinearIterpolationError}
\end{table}

\begin{table}[H]
\begin{center}
\begin{tabular}{|c |c c|c c|}
\hline
$h$    & $\|\mathbf{ u} - \mathbf{ u}_h\|_{0,\Omega}$ & rate   & $|\mathbf{ u} - \mathbf{ u}_h |_{1,\Omega}$ & rate   \\ \hline
1/10  & 6.6120E-1                  &                 & 6.8668E+0            &                 \\ \hline
1/20  & 1.6880E-1                 &  1.9698      & 3.4932E+0           &  0.9751      \\ \hline
1/40   & 4.2380E-2                 &  1.9938      & 1.7545E+0          & 0.9935  \\ \hline
1/80   & 1.0599E-2                 &  1.9995 & 8.7833E-1               & 0.9982 \\ \hline
1/160  & 2.6485E-3                 & 2.0007 & 4.3933E-1               & 0.9995 \\ \hline
1/320  & 6.6160E-4                 & 2.0011 & 2.1972E-1               & 0.9997 \\ \hline
1/640  & 1.6493E-4                 & 2.0041 & 1.0991E-1               & 0.9994 \\ \hline
1/1280 & 4.1100E-5                 & 2.0047 & 5.4990E-2               & 0.9990 \\ \hline
\end{tabular}
\end{center}
\caption{IFE Solution errors and rates for the bilinear IFE functions}
\label{table:bilinearSolutionError}
\end{table}

\bibliographystyle{plain}

\end{document}